\documentclass{amsart}
\title{Explicit Construction of Maass Wave Forms and Their Petersson Inner Products}
\author{Daichi Tanaka}
\usepackage[top=25mm, bottom=25mm, left=25mm, right=25mm]{geometry}
\usepackage{graphicx}
\usepackage{amsmath}
\allowdisplaybreaks[4]
\usepackage{amssymb}
\usepackage{amsthm}
\usepackage{amsfonts}
\usepackage{bm}
\usepackage{mathtools}
\usepackage{empheq}
\theoremstyle{definition}

\newtheorem{definition}{Definition}[section]
\newtheorem{proposition}[definition]{Proposition}
\newtheorem{lemma}[definition]{Lemma}
\newtheorem{theorem}[definition]{Theorem}
\newtheorem{corollary}[definition]{Corollary}
\newtheorem{remark}[definition]{Remark}

\numberwithin{equation}{section}
\date{}

\keywords{Maass wave forms, Petersson inner product, dihedral Artin representation}
\subjclass[2010]{}
\address{Daichi Tanaka \\ 
Mathematical Inst. Tohoku Univ.\\
 6-3,Aoba, Aramaki, Aoba-Ku, Sendai 980-8578, JAPAN}
\email{tanaka.daichi.t8@dc.tohoku.ac.jp}

\begin{document}
\begin{abstract}\noindent
  In this paper, we explicitly construct Maass wave cusp forms associated to 
  Hecke characters on arbitrary real quadratic fields.
  This result is a generalization of Maass (1949), who constructed 
  Maass wave cusp forms under the assumption that narrow class number is one.
  We also compute its Petersson inner product explicitly and 
  give a few examples involving dihedral Artin representations.
\end{abstract}
\maketitle
\tableofcontents
\section{Introduction}\label{sec;Introduction}

Constructing cusp forms on reductive groups is an important subject in both the theory of
automorphic representations and number theory.
There are many related works on 
holomorphic automorphic forms
\cite{Maass1979}, \cite{Yoshida1984}, \cite{Ikeda2001}.
However, the construction of 
non-holomorphic automorphic forms (cf.\cite{Niwa1991}, \cite{Miyazaki2004})
is less developed than holomorphic ones.
In this paper, we explicitly construct Maass wave forms from Hecke characters 
for real quadratic fields and compute their Petersson inner products. 

We prepare some notations to explain our results.
Let $N$ be a positive integer and
let $S^1$ be the multiplicative group of 
complex numbers with the absolute value $1$.
We denote 
$\Gamma_0(N)=\left\{\begin{pmatrix}
  a&b\\c&d
\end{pmatrix}\in SL_2(\mathbb{Z})\middle|\, N\mid c\right\}$. 
We take a Dirichlet character
$\chi\colon\left(\mathbb{Z}/N\mathbb{Z}\right)^\times\rightarrow S^1$ modulo $N$.
It is extended to $\mathbb{Z}$ by $\chi(n)=\chi(\overline{n})$ if $(n,N)=1$
and $\chi(n)=0$ otherwise.
A Maass wave form is a function $\Theta$ defined on the complex upper half plane 
$\mathbb{H}=\{z=x+iy\mid x,y\in \mathbb{R}, y>0\}$ with the
following properties:
\[
\left\{
\begin{aligned}
& (1)\,\Theta\left|\begin{pmatrix}
  a&b\\c&d
\end{pmatrix}\right.=\chi(d)\Theta\qquad\text{for any }
\begin{pmatrix}
  a&b\\c&d
\end{pmatrix}\in\Gamma_{0}(N),\\
& (2)\,\Theta\text{ is real analytic on }\mathbb{H}\text{ and an eigenfunction of }
\Delta=-y^{2}\left(\frac{\partial ^2}{\partial x^2}
+\frac{\partial^2}{\partial y^2}\right),\\
&\quad\text{ with eigenvalue }\frac{1}{4}\!-\nu^2\text{ where }\nu\in i\mathbb{R}. \\
& (3)\,\Theta\text{ is of moderate growth at each cusp of } \Gamma_{0}(N).
\end{aligned}
\right.
\]
Here, we denote $\left(\Theta\left|\begin{pmatrix}
  a&b\\c&d
\end{pmatrix}\right.\right)(z)
\coloneq \Theta\left(\displaystyle\frac{az+b}{bz+d}\right)$.
Let $M(\Gamma_{0}(N),\nu,\chi)$ be the space consisting of 
all Maass wave forms satisfying (1), (2), and (3).
The condition (3) means that at each cusp $c\in\mathbb{Q}\cup\{\infty\}$
if we take $\gamma\in GL_{2}(\mathbb{Q})^{+}$ such that
$\gamma(i\infty)=c$, then there exist
$t\,\in\mathbb{Z}_{\ge 1}$ and $C>0$ such that
\[\left(\Theta\left|\gamma^{-1}
\right.\right)(iy)\leq Cy^{t}
\qquad \text{for }y\text{ sufficiently large}.
\]
The condition (2) is an analogue to the holomorphic condition
(Cauchy-Riemann equation) in 
the definition of modular forms.
Moreover, if we denote
$\displaystyle\left(f\left|_{k}\begin{pmatrix}
  a&b\\c&d
\end{pmatrix}\right.\right)(z)
\coloneq (cz+d)^{-k}f\left(\frac{az+b}{bz+d}\right)$,
condition (1) becomes the automorphy condition in the usual definition 
of modular forms. Thus, Maass wave forms are an analogue of holomorphic
modular forms.
By condition (1), we have the Fourier expansion of a Maass wave form $\Theta$:
\[\Theta(z)=ay^{\nu+1/2}+by^{-\nu+1/2}+
\sum_{n\ne 0}a(n)\sqrt{y}K_{\nu}(2\pi |n|y)e^{2\pi inx}\]
where $a,b,a(n)\in\mathbb{C}$.
Just as we can define modular cusp forms in the theory of holomorphic modular forms, we can also define Maass wave cusp form by 
a function on $\mathbb{H}$ satisfying conditions (1), (2) and the following condition $(3)'$: 
\vspace{5mm}\\
$(3)'$\,$\Theta$ decays rapidly at each cusps. Namely, for any cusp
$c\in \mathbb{Q}\cup \{\infty\}$, 
$\gamma\in GL_2(\mathbb{Q})^+$ with $\gamma(i\infty)=c$ and
$t\in\mathbb{Z}$,
there exists $C>0$ such that
$\left(\Theta\left|\gamma^{-1}\right.\right)(iy)\leq Cy^{-t}$
for $y$ sufficiently large.
\vspace{5mm}\\
We denotes by $S(\Gamma_{0}(N),\nu,\chi)$ 
the space consisting of all Maass wave cusp forms satisfying conditions (1), (2), and $(3)'$.
By the condition $(3)'$, $\Theta\in S(\Gamma_{0}(N),\nu,\chi)$ can be expressed as
\begin{equation}\label{eq;Fourier expansion of theta}
\Theta(z)=\sum_{n\in \mathbb{Z}\backslash{\{0\}}}a(n)\sqrt{y}K_{\nu}(2\pi |n|y)e^{2\pi inx}.
\end{equation}
As explained in Section \ref{sec;main result1}, we construct Maass wave cusp forms associated Hecke characters
on a real quadratic field $F$.
We denote by $J^{\mathfrak{f}}_{F}$ the group consisting of 
non-zero fractional ideals in $F$.
Let $\psi\colon J^{\mathfrak{f}}_{F}\to S^1$ be a Hecke character 
modulo $\mathfrak{f}$ on $F$ (see Section \ref{subsec;Hecke character}).
We assume that $\psi$ is
of the type $(\epsilon,\epsilon,\frac{\nu}{i},-\frac{\nu}{i})$,
that is, $\psi((a))=\psi_{\mathrm{fin}}(a)\psi_{\infty}(a)$ for any $a\in \mathcal{O}_{F}$
where $\psi_{\mathrm{fin}}$ is a homomorphism 
$\left(\mathcal{O}_{F}/\mathfrak{f}\right)^\times\to S^1$, and 
$\psi_{\infty}(a)=\left(\frac{a^{(1)}}{|a^{(1)}|}\right)^\epsilon
\left(\frac{a^{(2)}}{|a^{(2)}|}\right)^\epsilon |a^{(1)}|^\nu |a^{(2)}|^{-\nu}$
($\epsilon\in\{0,1\}$, $\nu\in i\mathbb{R}$).
Note that $a^{(1)}$, $a^{(2)}$ mean the real embeddings of $a\in \mathcal{O}_F$.  
In view of the Fourier expansion in Equation (\ref{eq;Fourier expansion of theta}), 
we formally define the function $\Theta_{\psi}$ given by 
\begin{equation*}
		\Theta_{\psi}(z)=
		\begin{cases}
			\displaystyle \sum_{\mathfrak{a}}\psi(\mathfrak{a})\sqrt{y}K_{\nu}(2\pi\mathbb{N}_{F/\mathbb{Q}}(\mathfrak{a})y)
			\cos (2\pi\mathbb{N}_{F/\mathbb{Q}}(\mathfrak{a})x)\quad&(\text{if }\epsilon =0),\\
			\displaystyle \sum_{\mathfrak{a}}\psi(\mathfrak{a})\sqrt{y}K_{\nu}(2\pi\mathbb{N}_{F/\mathbb{Q}}(\mathfrak{a})y)
			\sin (2\pi\mathbb{N}_{F/\mathbb{Q}}(\mathfrak{a})x)&(\text{if }\epsilon =1)
		\end{cases}
\end{equation*}
where $\mathbb{N}_{F/\mathbb{Q}}(\mathfrak{a})=\# \left(\mathcal{O}_F/\mathfrak{a}\right)$
and $\mathfrak{a}$ runs over the integral ideals of $F$.
Then, our first result is the following.
\begin{theorem}\label{main1}
  Let $F$ be a real quadratic field and let $\psi$ be a primitive Hecke character
  modulo $\mathfrak{f}$ on $F$. 
  We denote by $D$ the discriminant of $F$.
  We assume that $\psi$ is of the type $(\epsilon,\epsilon,\frac{\nu}{i},-\frac{\nu}{i})$ and not of the form 
  $\mathbb{N}_{F/\mathbb{Q}}\circ\chi$ for any Dirichlet character $\chi$.
  Then, $\Theta_{\psi}$  belongs to
  $ S(\Gamma_{0}(D\mathbb{N}_{F/\mathbb{Q}}(\mathfrak{f})),\nu, \chi_{D}\psi_{\mathrm{fin}})$.
\end{theorem}
As for the level, we note that if 
$\mathfrak{f}$ and 
$D$ are not coprime, then the resulting Maass wave forms have a lower level,  namely
the least common multiple of $D$ and $\mathbb{N}_{F/\mathbb{Q}}(\mathfrak{f})$.
We also note that $\Theta_{\psi}$ is a Hecke eigenform since
the associated $L$-function is the Hecke $L$-function $L(s,\psi)$.
The proof of Theorem \ref{main1} is given in Section \ref{sec;main result1}. The method is 
similar to the one used in \cite[p112]{Bump1997}.
However, we show a more general assertion 
(see Theorem \ref{main2}).
Maass \cite{Maass1949} proved Theorem \ref{main1} for real quadratic fields $F$ 
of narrow class number one and this assumption forces $\nu\ne 0$ for the constructed Maass wave forms to be cuspidal.
Thus, his case excludes the Maass wave forms with $\nu =0$ which correspond to 
dihedral Artin representations.
Our Theorem \ref{main1} extends Maass's result to any real quadratic field $F$ and also include the case when $\nu=0$.\\
The latter part of this paper is on 
computation of Petersson inner product for Maass wave cusp forms.
We define Petersson inner products for Maass wave cusp forms 
similarly to that of modular cusp forms.
\begin{definition}
  Let $\Theta_1$ and $\Theta_2$ be elements of $S(\Gamma_0(N),\nu,\chi)$.
  Then the Petersson inner product $\left<\Theta_1,\Theta_2\right>$
  is defined as follows:
  \[\left<\Theta_1,\Theta_2\right>=
  \int_{\Gamma_0(N)\backslash \mathbb{H}}\Theta_1(z)\overline{\Theta_2(z)}\frac{dxdy}{y^2}.\]
\end{definition}

By Rankin-Selberg method and combinatorial computations
involving Satake parameters,
we have an explicit formula for Petersson inner products 
of Maass wave cusp forms constructed in Theorem \ref{main1}.
\begin{theorem}\label{theorem;explicit computation of Petersson inner product}
  Let $\Theta_{\psi}$ be the Maass wave cusp form in Theorem \ref{main1}.
  Then it follows that
  \begin{equation}
  \left<\Theta_{\psi},\Theta_{\psi}\right>
  =C_1C_2C_3\cdot\mathrm{Res}_{s=1}(\zeta_{F}(s))L(1,\psi(\overline{\psi}\circ\sigma))
  \end{equation}
  where
  \begin{itemize}
      \item $\displaystyle C_1=\frac{1}{4\pi}\phi(D\mathbb{N}_{F/\mathbb{Q}} (\mathfrak{f}))^{-1}(D\mathbb{N}_{F/\mathbb{Q}}(\mathfrak{f}))^{2}$,\\
      \item $\displaystyle C_2=
            \Gamma\left(\frac{1+2\nu}{2}\right)
            \Gamma\left(\frac{1-2\nu}{2}\right)$,\\
      \item $\displaystyle C_3=\left(
    \prod_{\substack{\mathfrak{p}\colon\text{split},\,\mathfrak{p}\nmid\mathfrak{f}\\
    \mathbb{N}_{F/\mathbb{Q}}(\mathfrak{p})\mid\mathbb{N}_{F/\mathbb{Q}}(\mathfrak{f})}}
    (1-\mathbb{N}_{F/\mathbb{Q}}(\mathfrak{p})^{-1})^{-1}\right)
    \times\left(\prod_{p\mid D\mathbb{N}_{F/\mathbb{Q}}(\mathfrak{f})}(1-p^{-1})(1-\chi_{D}(p)p^{-1})\right).$
  \end{itemize}
%  \frac{1}{4\pi}\phi(D\mathbb{N}_{F/\mathbb{Q}}(\mathfrak{f}))^{-1}(D\mathbb{N}_{F/\mathbb{Q}}(\mathfrak{f}))^{2}
 %  \Gamma\left(\frac{1+2\nu}{2}\right)
%   \Gamma\left(\frac{1-2\nu}{2}\right)\\
 %  &\times\left(   \prod_{\substack{\mathfrak{p}\colon\text{split},\,\mathfrak{p}\nmid\mathfrak{f}\\
  %  \mathbb{N}_{F/\mathbb{Q}}(\mathfrak{p})\mid\mathbb{N}_{F/\mathbb{Q}}(\mathfrak{f})}}
   % (1-\mathbb{N}_{F/\mathbb{Q}}(\mathfrak{p})^{-1})^{-1}\right)
   % \times\left(\prod_{p\mid D\mathbb{N}_{F/\mathbb{Q}}(\mathfrak{f})}(1-p^{-1})(1-\chi_{D}(p)p^{-1})\right)\\
   %&\times\mathrm{Res}_{s=1}(\zeta_{F}(s))L(1,\psi(\overline{\psi}\circ\sigma)).
  %\end{align*}
  Here, $\sigma$ is the nontrivial element
  of $\text{Gal}(F/\mathbb{Q})$ and $\phi$ is the Euler totient function.
\end{theorem}

In \cite[Section 6]{Stark1975}, Stark gave explicit computation of 
Petersson inner product of 
the modular cusp form defined from a dihedral Artin representation.
Theorem \ref{theorem;explicit computation of Petersson inner product} 
gives a real quadratic analogue of the Stark's result.
Finally, we give some examples to compute Petersson inner products in terms of arithmetic invariants (see Section \ref{SomeExamples} for details).  
\begin{corollary}\label{corollary;computation of Petersson inner product}
  Let $F$ be a real quadratic field, and let 
  $\psi$ be the Hecke character of $J^{(1)}_{F}$ associated to  a primitive character $\mathcal{C}_F\rightarrow S^1$ of the ideal class group $\mathcal{C}_F$ of $F$ via the natural surjection $J^{(1)}_{F}\longrightarrow \mathcal{C}_F$.
  \begin{enumerate}
    \item If $F=\mathbb{Q}(\sqrt{229})$, then $\mathcal{C}_F\cong \mathbb{Z}/3\mathbb{Z}$
          and the Hecke character
          $\psi$ is of the type $(0,0,0,0)$ and not of the form 
          $\chi\circ\mathbb{N}_{F/\mathbb{Q}}$.
          Let $K=\mathbb{Q}(\alpha)$ where $\alpha$ is a root of $X^3-4X-1$.
          Then the Petersson inner product for $\Theta_\psi$ is given by
          \[\langle\Theta_{\psi},\Theta_{\psi}\rangle=6R_{K}\cdot R_{F}\]
          where $R_F$, $R_K$ are the regulators of $F$, $K$ respectively.          
    \item If $F=\mathbb{Q}(\sqrt{445})$, then $\mathcal{C}_F\cong \mathbb{Z}/4\mathbb{Z}$ 
          and the Hecke character $\psi$ is of the type $(0,0,0,0)$ and 
          not of the form $\chi\circ\mathbb{N}_{F/\mathbb{Q}}$.
          Let $K_2=\mathbb{Q}(\sqrt{5},\sqrt{89})$.
          Then the Petersson inner product for $\Theta_\psi$ is given by
          \[\langle\Theta_{\psi},\Theta_{\psi}\rangle=4R_{K_2}\]
          where $R_{K_2}$ is the regulator of $K_2$. 
    \item If $F=\mathbb{Q}(\sqrt{401})$, then $\mathcal{C}_F\cong\mathbb{Z}/5\mathbb{Z}$
          and the Hecke character
          $\psi$ is of the type $(0,0,0,0)$ and not of the form 
          $\chi\circ\mathbb{N}_{F/\mathbb{Q}}$.
          Let $K=\mathbb{Q}(\alpha)$ where $\alpha$ is a root of 
          $X^5-X^4-5X^3+4X^2+3X-1$. 
          Then the following relation hold:
          \[\langle\Theta_{\psi},\Theta_{\psi}\rangle\cdot
          \langle\Theta_{\psi^2},\Theta_{\psi^2}\rangle=100{R_{F}}^2\cdot R_{K}\]
          where $R_F$, $R_{K}$ are the regulators of $F$ and $K$ respectively. 
  \end{enumerate} 
\end{corollary}
The algebraic independence of special value of $L$-functions 
are studied in 
\cite{MurtyMurty2011}, \cite{GunMurty2011}.
Corollary \ref{corollary;computation of Petersson inner product}
shows how the transcendental information of $L$-functions relate to regulators
through Petersson inner products.
Another important aspect of the Petersson inner product is 
that it reflects geometric information of the modular curve.
For modular cusp forms, a computation was given by Zagier \cite{Zagier1985}.
Moreover, alternative formula was derived in \cite{Martin2005} for weight $k=1$, 
and in \cite{Haberland1983} and \cite{CohenRoblot1999} for weight $k\ge 2$.
We hope that our computation is used to study the $L$-values of Maass wave
cusp forms. Another possible application is the study of the $L^4$-norms of our Maass wave cusp forms, as in \cite{Luo}. This will be the subject of our future work.

The organization of this paper is as follows.
In Section \ref{sec;Hecke characters and Hecke L functions}, 
we recall some properties of Hecke characters and Hecke 
$L$-functions.
Section \ref{sec;main result1} is devoted to
the construction of Maass wave forms .
Our argument is largely based on \cite[p.112]{Bump1997}.
However, there is an error in the proof of the cuspidality 
(see Remark \ref{remark;pf of main result}), 
so we have to employ a different method. 
In Section \ref{sec;Petesson inner product}, 
we derive an explicit formula for the Petersson inner product 
of the constructed Maass wave cusp forms.
After that, we give concrete examples using class field theory,
and see that Petersson inner products
can be expressed in terms of arithmetic invariants or special values of
dihedral Artin $L$-functions.

\subsection*{Acknowledgements}
The author would like to express my sincere gratitude to my supervisor, 
Professor Takuya Yamauchi, for his invaluable guidance, numerous comments, 
and constant encouragement throughout this work.
He is also grateful to Professor Henry H. Kim  
for explaining the method of computation for Corollary 
\ref{corollary;computation of Petersson inner product}.
Finally, he would like to thank Eisuke Otsuka, 
a senior member of my laboratory, for his helpful advice on academic writing.

\section*{Notation}
\begin{itemize}
  \item $\mathbb{H}=\left\{z=x+iy\mid x\in\mathbb{R},\,y>0\right\}$.
  \item For $\gamma=\begin{pmatrix}
    a&b\\c&d
  \end{pmatrix}\in GL_2(\mathbb{R})^{+}$ and $z\in \mathbb{H}$,
  we denote $\displaystyle\gamma z=\frac{az+b}{cz+d}$.
  \item $\Delta=-y^{2}\left(\displaystyle
  \frac{\partial^2 }{\partial x^2}+\frac{\partial^2}{\partial y^2}\right)$.
  \item $S^1=\{z\in \mathbb{C}\mid\, |z|=1\}$.
  \item For $N$ be a positive integer, we denote 
  $\Gamma_{0}(N)=\left\{\begin{pmatrix}
    a&b\\c&d
  \end{pmatrix}\in SL_2(\mathbb{Z})\middle|\, N\mid c \right\}$.   
  \item Let $K$ be a number field and $\mathfrak{f}$ an integral ideal of $K$.
       We denote the group consisting of non-zero fractional ideals of $K$ coprime to $\mathfrak{f}$ 
       by $J^{\mathfrak{f}}_{K}$.
  \item Let $\nu$ be a pure imaginary number.
        we denote by $K_{\nu}$ the $K$-Bessel function. An integral expression of $K_{\nu}$ is given by
        $K_{\nu}(y)=\displaystyle
        \frac{1}{2}\int_{0}^{\infty}e^{-y\frac{t+t^{-1}}{2}}t^\nu\frac{dt}{t}$
        for $y>0$.
  \item For a number field $K$, we denote by $S_{K}$ (resp. $S'_{K}$)
        the set of real (resp. imaginary) places
        of $K$. We also denote by 
        $K_{\infty}^\times$
        the infinite part of idele of $K$,
        i.e.
        $\displaystyle K_{\infty}^\times=\prod_{v\in S_{K}}\mathbb{R}\times\prod_{v\in S'_{K}}\mathbb{C}$. 
\end{itemize}
\section{Hecke characters and Hecke $L$-functions}\label{sec;Hecke characters and Hecke L functions}
In this section, we describe the notion of Hecke character and its associated
$L$-function. 
In Section \ref{subsec;Hecke character},
we review basics of
Hecke character and its Gauss sum.
We refer to \cite{Neukirch1992} or \cite{Raghuram2022}.
In Section \ref{subsec;Lfunction}, we review well-known
analytic properties of Hecke $L$-functions.
As an application, we derive 
Lemma \ref{lemma;relation in Gauss sums},
which gives a relation
for Gauss sums
of Hecke characters defined via the norm map on real quadratic fields. 
Section \ref{subsec;Lfunction} is based on
\cite{Miyake2006}.
\subsection{Hecke characters}\label{subsec;Hecke character}
Let $K$ be a number field, and $\mathfrak{f}$ be a integral ideal
of $K$. Define
\[J_{K}^{\mathfrak{f}}=
\{I:\text{fractional ideal of }K|\,I \text{ is prime to }\mathfrak{f}\}.\]

\begin{definition}
Let $S^1=\{z\mid |z|=1\}$ and let $\psi$ be a homomorphism $J_{K}^{\mathfrak{f}}\rightarrow S^1$.
We call $\psi$ a Hecke character modulo $\mathfrak{f}$ if there exist
homomorphism $\psi :\left(\mathcal{O}_K/\mathfrak{f}\right)^\times\rightarrow S^1$
and continuous homomorphism
$\psi_\infty :K_{\infty}^\times \rightarrow S^1$ such that 
\[\psi \left((a)\right)=\psi_{\mathrm{fin}}(a)\psi_\infty(a) \]
for any $a\in \mathcal{O}_K$ with $(a,\mathfrak{f})=1$.
Here 
$\displaystyle K_{\infty}^\times=\prod_{v\in S_{K}}
\mathbb{R}\times \prod _{v\in S'_{K}}\mathbb{C}$
where $S_{K}$ (resp. $S'_{K}$) 
is the set of real (resp. imaginary) places of $K$. 
\end{definition} 
\begin{remark}
  For a Hecke character $\psi$ modulo $\mathfrak{f}$, 
  $\psi_{\mathrm{fin}}$ and $\psi_\infty$ in above are 
  unique. This is because elements $a\in K$ with $a-1\in\mathfrak{f}$
  is dense in $K_{\infty}^\times$.
\end{remark}
Next we show fundamental properties of $\psi_{\mathrm{fin}}$.
\begin{theorem}(\cite[p472]{Neukirch1992})\label{conductor}
  Let $\psi$ be a Hecke character modulo $\mathfrak{f}$ and let $\mathfrak{f}'$ be
  a proper divisor ideal of $\mathfrak{f}$.
  Then the following conditions are equivalent.
  \begin{itemize}
    \item [(1)]$\psi$ is a restriction of some 
    Hecke character $\psi ':J^{\mathfrak{f}'}\rightarrow S^1\mod \mathfrak{f}'$.
    (Note that $J_{K}^{\mathfrak{f}}\subset J_{K}^{\mathfrak{f}'}$).
    \item [(2)]There exists a homomorphism
    $\psi '_{\mathrm{fin}}:\left(\mathcal{O}_K/\mathfrak{f}'\right)^\times\rightarrow S^1$
    such that $\psi _{\mathrm{fin}}=\psi '_{\mathrm{fin}}\circ \pi$.
    Here we denote by $\pi$ the canonical map 
    $\left(\mathcal{O}_K/\mathfrak{f}\right)^\times
    \rightarrow \left(\mathcal{O}_K/\mathfrak{f}'\right)^\times$.
  \end{itemize}
\end{theorem}
\begin{proof}
  (1)$\Rightarrow$(2)
  Let
  \begin{align*}
  &\widetilde{\psi}_{\mathrm{fin}}:\left(\mathcal{O}_{K}/\mathfrak{f}\right)^\times
  \to\left(\mathcal{O}_{K}/\mathfrak{f}'\right)^\times
  \xrightarrow{\psi '_{\mathrm{fin}}}S^{1}\\
  &\widetilde{\psi}_{\infty}=\psi '_{\infty}:K_{\infty}^\times \to S^{1}.
  \end{align*}
  Since $\psi((a))=\psi'((a))=\psi '_{\mathrm{fin}}(a)\psi '_{\infty}(a)=
  \widetilde{\psi}_{\mathrm{fin}}(a)\widetilde{\psi}_{\infty}(a)$ for any $(a)\in P_{K}^{\mathfrak{f}}$,
  the uniqueness of $\psi'_{\infty}$ and $\psi '_{\mathrm{fin}}$ implies
  $\psi_{\mathrm{fin}}=\widetilde{\psi}_{\mathrm{fin}}$,
  $\psi_{\infty}=\widetilde{\psi}_{\infty}$.\\
  (2)$\Rightarrow$(1)\\
  Define $\psi_{\mathrm{fin}}$ by the following composition of maps:
  \[\left(\mathcal{O}_{K}/\mathfrak{f}\right)^\times\to
  \left(\mathcal{O}_{K}/\mathfrak{f}'\right)^\times
  \xrightarrow{\psi'_{\mathrm{fin}}}S^1.\]
 The characters $\psi_{\mathrm{fin}}$ and $\psi '_{\mathrm{fin}}$ are 
  extended to $K^\mathfrak{f}$ and $K^{\mathfrak{f}'}$ respectively
  where $K^{\mathfrak{f}}=
  \{\frac{a}{b}\mid a,b\in \mathcal{O}_{K}\backslash \{0\},(ab,\mathfrak{f})=1\}$
  $K^{\mathfrak{f}'}$ is defined similarly.
  Then $\psi_{\mathrm{fin}}$ is the restriction of 
  $\psi'_{\mathrm{fin}}$ to $K^{\mathfrak{f}'}$.
  For $\mathfrak{a}'\in J^{\mathfrak{f}'}$, we take
  $a\in K^{\mathfrak{f}'}$ such that 
  $(\mathfrak{a}'a^{-1},\mathfrak{f})=1$.
  Put $\mathfrak{a}=\mathfrak{a}'a^{-1}$ and
  set $\psi '(\mathfrak{a}')=\psi(\mathfrak{a})\psi '_{\mathrm{fin}}(a)
  \psi_{\infty}(a)$.
  This definition is independent on choice of $a$. In fact, 
  if we choose another $a_{1}\in K^{\mathfrak{f}'}$ as above and
  put $\mathfrak{a}_{1}=\mathfrak{a}'a_{1}^{-1}$, we have
  \begin{align*}
    \psi(\mathfrak{a})\psi '_{\mathrm{fin}}(a)\psi_{\infty}(a)
    =&\psi(\mathfrak{a}_{1})\psi '_{\mathrm{fin}}(a_{1})\psi_{\infty}(a_{1})
    \psi(\mathfrak{a}a_{1}^{-1})\psi '_{\mathrm{fin}}(aa_{1}^{-1})\psi_{\infty}(aa_{1}^{-1})\\
    =&\psi(\mathfrak{a}_{1})\psi '_{\mathrm{fin}}(a_{1})\psi '_{\mathrm{fin}}(a_{1})
    \psi(a_{1}a^{-1})\psi_{\mathrm{fin}}(aa_{1}^{-1})\psi_{\infty}(aa_{1}^{-1})\\
    =&\psi(\mathfrak{a}_{1})\psi '_{\mathrm{fin}}(a_{1})\psi_{\infty}(a_{1}).  
  \end{align*}
  By the definition, it is clear that the restriction of 
  $\psi '$ to $J_{K}^{\mathfrak{f}}$ is equal to $\psi$.
  Also, for $a'\in P_{K}^{\mathfrak{f}'}$ (put $a'=ab$, 
  $(a)\in J_{K}^{\mathfrak{f}'}$, $(b)\in P_{K}^{\mathfrak{f}}$), we have
  \begin{align*}
    \psi '((a'))=&
    \psi ((b))\psi_{\mathrm{fin}} '(a)\psi_{\infty}(a)
    =\psi_{\mathrm{fin}}(b)\psi_{\infty}(b)\psi'_{\mathrm{fin}}(a)\psi_{\infty}(a)\\
    =&\psi '_{\mathrm{fin}}(ab)\psi_{\infty}(ab)=\psi_{\mathrm{fin}} '(a')\psi_{\infty}(a').
  \end{align*}
  Thus, $\psi '$ is a Hecke character modulo $\mathfrak{f}'$.
\end{proof}
\begin{definition}
  Let $\psi$ be a Hecke character modulo $\mathfrak{f}$.
  We say $\psi$ is primitive if $\psi$
  does not satisfy the equivalent conditions in Theorem \ref{conductor}
  for any proper divisor $\mathfrak{f}'$ of $\mathfrak{f}$.
\end{definition}
\begin{definition}(Gauss Sum)\label{definition of Gauss sum}
  Let $\mathfrak{d}$ be the different of $K/\mathbb{Q}$.
  \begin{enumerate}
    \item Let $\psi_{\mathrm{fin}}$ be a character of 
    $\left(\mathcal{O}_K/\mathfrak{f}\right)^\times$ and
  let $y\in \mathfrak{f}^{-1}\mathfrak{d}^{-1}$.
  Then, we define Gauss sum of $\psi_{\mathrm{fin}}$ by
  \[\tau _{K}\left(\psi_{\mathrm{fin}},y\right)
  =\sum_{\substack{x\bmod\mathfrak{f}\\(x,\mathfrak{f})=1}}\psi_{\mathrm{fin}}(x)
  e^{2\pi i\mathrm{Tr}(xy)},\]
    \item 
  Let $\psi=\psi_{\mathrm{fin}}\psi_{\infty}$ be a 
  primitive Hecke character modulo $\mathfrak{f}$.
  Choose an integral ideal $\mathfrak{a}$ of $K$  
  such that $\mathfrak{dfa}$ is a principal ideal and 
  $(\mathfrak{f},\,\mathfrak{a})=1$ and write 
  $\mathfrak{dfa}=(a)$ where $a\in K$.
  Then Gauss sum $\tau_{K}(\psi)$ of $\psi$ is defined by
  \[\tau_{K}(\psi)
  =\frac{\psi_{\infty}(a)}{\psi(\mathfrak{a})}
  \sum_{x\in \mathfrak{a}\bmod \mathfrak{fa}}
  \psi_{\mathrm{fin}}(x)e^{2\pi i\mathrm{Tr}(x/a)}.\]
  Note that this definition does not depend
  on the choice of $\mathfrak{a}$ and $a$.
  \end{enumerate}
\end{definition}
\begin{remark}
  Let $\chi$ be a Dirichlet character modulo $l$.
  Then $\chi$ can be regarded as a Hecke character $\psi$ on $\mathbb{Q}$ mod $l\mathbb{Z}$. 
  We therefore denote Gauss sum $\tau_{\mathbb{Q}}(\psi)$ by $\tau_{\mathbb{Q}}(\chi)$.
\end{remark}
All we need are the properties of $\tau_{K}(\psi)$,
However, in order to understand $\tau_{K}(\psi)$,
we have to study properties of $\tau_{K}(\psi_{\mathrm{fin}},y)$.
\begin{theorem}{\cite[p.473]{Neukirch1992}}\label{absolute value of Gauss1}
  Let $\psi_{\mathrm{fin}}$ be a primitive character
  of $\left(\mathcal{O}_K/\mathfrak{f}\right)^\times$,
  and let $y\in\mathfrak{f}^{-1}\mathfrak{d}^{-1}$ and
  $a\in\mathcal{O}_K$.
  Then we have
  \[\tau_{K}\left(\psi_{\mathrm{fin}},ay\right)=
  \begin{cases}
            \overline{\psi}_{\mathrm{fin}}(a)\tau_{K}\left(\psi_{\mathrm{fin}},y\right)  
            \quad&\text{if }(a,\mathfrak{f})=1\\       
            0\qquad &\text{if }(a,\mathfrak{f})\ne 1.
  \end{cases}\]
  Moreover, if $\left(y\mathfrak{f}\mathfrak{d},\mathfrak{f}\right)=1$
  then
  \[\left|\tau_{K}\left(\psi_{\mathrm{fin}},y\right)\right|
  =\sqrt{\mathbb{N}_{K/\mathbb{Q}}(\mathfrak{f})}.\]
\end{theorem}

Theorem \ref{absolute value of Gauss1} needs several lemmas.

\begin{definition}
  Let $\mathfrak{a}$ be an integral ideal with prime ideal decomposition
  $\mathfrak{a}=\mathfrak{p}_{1} ^{\nu _1}\cdots\mathfrak{p}_{r} ^{\nu _r}$.
  We define Möbius function $\mu(\mathfrak{a})$ by
  \[\mu (\mathfrak{a})=
  \begin{cases}
    1 \qquad      &\text{if }r=0 \;\text{i.e.}\,\mathfrak{a}=(1)\\
    (-1)^r \qquad &\text{if }r\geq 1,\nu_1=\cdots =\nu_r=1\\
    0\qquad       &\text{otherwise}.
  \end{cases}\]
\end{definition}
\begin{proposition}(\cite[p.474]{Neukirch1992})\label{prop;Mobius}
  Suppose that $\mathfrak{a}\ne1$. Then 
  $\displaystyle\sum_{\mathfrak{b}|\mathfrak{a}}\mu(\mathfrak{b})=0$.
\end{proposition}
\begin{proof}
  Put $\mathfrak{a}=\mathfrak{p}_{1}^{\nu_{1}}\cdots\mathfrak{p}_{r}^{\nu_{r}}$.
  Then we compute as follows: 
  \begin{align*}
    \sum_{\mathfrak{b}\mid\mathfrak{a}}\mu(\mathfrak{b})
    =&\mu(1)+\sum_{1=1}^{r}\mu(\mathfrak{p}_{i})+
    \sum_{i_{1}<i_{2}}\mu(\mathfrak{p}_{1_{1}}\mathfrak{p}_{i_{2}})
    +\cdots +\mu(\mathfrak{p}_{1}\cdots\mathfrak{p}_{r})\\
    =&1-\binom{r}{1}+\binom{r}{2}-\cdots+(-1)^{r}\binom{r}{r}
    =(1+(-1))^{r}=0.
  \end{align*}
\end{proof}
For $y\in \mathfrak{f}^{-1}\mathfrak{d}^{-1}$ and 
any divisor $\mathfrak{a}$ of $\mathfrak{f}$, we set
\[T_{\mathfrak{a}}(y)=\sum_{\substack{x\bmod\mathfrak{f}\\(x,\mathfrak{f})=\mathfrak{a}}}
e^{2\pi i\mathrm{Tr}(xy)},\quad S_{\mathfrak{a}}(y)=
\sum_{\substack{x\bmod\mathfrak{f}\\\mathfrak{a}|x}}e^{2\pi i\mathrm{Tr}(xy)}.\]
Under these assumption, we can show the following lemma.
\begin{lemma}(\cite[p.474]{Neukirch1992})
  It follows that
  \begin{enumerate}
    \item $\displaystyle{T_1(y)=\sum_{\mathfrak{a}|\mathfrak{f}}\mu (\mathfrak{a})
    S_\mathfrak{a}(y)}$,
    \item $\displaystyle{S_\mathfrak{a}(y)=\begin{cases}
      \mathbb{N}_{K/\mathbb{Q}}\left(\frac{\mathfrak{f}}{\mathfrak{a}}\right)
      \qquad &\text{if }y\in \mathfrak{a}^{-1}\mathfrak{d}^{-1}\\
      0\qquad &\text{if }y\notin \mathfrak{a}^{-1}\mathfrak{d}^{-1}.
    \end{cases}}$
  \end{enumerate}
\end{lemma}
\begin{proof}
  Applying Proposition \ref{prop;Mobius}, we have 
  \begin{align*}
    \sum_{\mathfrak{a}\mid\mathfrak{f}}\mu(\mathfrak{a})S_{\mathfrak{a}}(y)
    =\sum_{\mathfrak{a}\mid \mathfrak{f}}\mu(\mathfrak{a})
    \sum_{\substack{\mathfrak{b}\\\mathfrak{a}\mid\mathfrak{b}\mid\mathfrak{f}}}
    T_{\mathfrak{b}}(y)
    =\sum_{\mathfrak{b}\mid\mathfrak{f}}T_{\mathfrak{b}}(y)
    \sum_{\mathfrak{a}\mid\mathfrak{b}}\mu(\mathfrak{b})
    =\sum_{\mathfrak{b}\mid\mathfrak{f}}T_{\mathfrak{b}}(y)\cdot\begin{cases*}
      1\quad\text{if } \mathfrak{b}=1\\0\quad\text{if } \mathfrak{b}\ne 1
    \end{cases*}
    =T_{1}(y).
  \end{align*}
  If $y\in \mathfrak{a}^{-1}\mathfrak{d}^{-1}$ and $x\in \mathfrak{a}$,
  then $xy\in \mathfrak{d}^{-1}$, hence $\mathrm{Tr}(xy)\in\mathbb{Z}$.  
  Therefore,
  \begin{align*}
    S_{\mathfrak{a}}(y)=\#\left\{\overline{x}\in \mathcal{O}_{K}/\mathfrak{f} \mid\,
    x\in\mathfrak{a}\right\}
    =\#\left(\mathfrak{a}/\mathfrak{f}\right)
    =\mathbb{N}_{K/\mathbb{Q}}\left(\frac{\mathfrak{f}}{\mathfrak{a}}\right).
  \end{align*}
  On the other hand, if $y\ne\mathfrak{a}^{-1}\mathfrak{d}^{-1}$,
  there exists $z\in\mathfrak{a}$
  such that $\mathrm{Tr}(zy)\notin \mathbb{Z}$.
  Thus, 
  \begin{align*}
    e^{2\pi i\mathrm{Tr}(zy)}S_{\mathfrak{a}}(y)
    =\sum_{\substack{x\bmod \mathfrak{f}\\\mathfrak{a}\mid x}}
    e^{2\pi i\mathrm{Tr}((x+z)y)}
    =S_{\mathfrak{a}}(y).
  \end{align*}
  Since $e^{2\pi i\mathrm{Tr}(zy)}\ne 1$, we obtain $S_{\mathfrak{a}}(y)=0$.
\end{proof}
\vspace{10pt}
\begin{proof}(Theorem \ref{absolute value of Gauss1})
  Let $a\in\mathcal{O}_K,(a,\mathfrak{f})=1$. If $x$ runs over
  representatives of $\left(\mathcal{O}_K/\mathfrak{f}\right)^\times$,
  $xa$ also runs over representatives of $\left(\mathcal{O}_K/\mathfrak{f}\right)^\times$.
  Therefore
  \begin{align*}
    \tau_{K}(\psi_{\mathrm{fin}},ay)
    =\sum_{\substack{x\bmod\mathfrak{f}\\(x,\mathfrak{f})=1}}
    \psi_{\mathrm{fin}}(x)e^{2\pi i\mathrm{Tr}(xay)}
    =\overline{\psi}_{\mathrm{fin}}(a)\tau_{K}(\psi_{\mathrm{fin}},y).
  \end{align*}
Next we consider the case of $(a,\mathfrak{f})=\mathfrak{f}_1\ne 1$.
Since $\psi_{\mathrm{fin}}$ is primitive,
we can take 
$b\bmod\mathfrak{f}\in\left(\mathcal{O}_K/\mathfrak{f}\right)^\times$ with
$\psi_{\mathrm{fin}}(b)\ne 1,b\equiv 1,\bmod\frac{\mathfrak{f}}{\mathfrak{f}'}$
such that
$b\bmod\mathfrak{f}\in\left(\mathcal{O}_K/\mathfrak{f}\right)^\times$.
Since
$aby-ay\in\mathfrak{d}^{-1}$, we have
\begin{align*}
  \overline{\psi}_{\mathrm{fin}}(b)\tau_{K}(\psi_{\mathrm{fin}},ay)
  =&\tau_{K}(\psi_\mathrm{fin},aby)
  =\sum_{\substack{x\bmod\mathfrak{f}\\(x,\mathfrak{f})=1}}\psi_{\mathrm{fin}}(x)
  e^{2\pi i\mathrm{Tr}(xaby)}\\
  =&\sum_{\substack{x\bmod\mathfrak{f}\\(x,\mathfrak{f})=1}}\psi_{\mathrm{fin}}(x)
  e^{2\pi i\mathrm{Tr}(xaby)}
  =\tau_{K}(\psi_{\mathrm{fin}},ay).
\end{align*}
This shows $\tau_{K}(\psi_{\mathrm{fin}},ay)=0$.\\
Finally we compute the absolute value of
$\tau(\psi_{\mathrm{fin}},y)$ when $(y\mathfrak{fd},\mathfrak{f})=1$.
\begin{align*}
  |\tau_{K}(\psi_{\mathrm{fin}},y)|^2
  =&\sum_{\substack{x\bmod\mathfrak{f}\\(x,\mathfrak{f})=1}}\tau_{\mathfrak{f}}(\psi_{\mathrm{fin}},y)
  \overline{\psi}_{\mathrm{fin}}(x)e^{-2\pi i \mathrm{Tr}(xy)}
  =\sum_{\substack{x\bmod\mathfrak{f}\\(x,\mathfrak{f})=1}}\tau_{\mathfrak{f}}(\psi_{\mathrm{fin}},xy)
  \overline{\psi}_{\mathrm{fin}}(x)e^{-2\pi i \mathrm{Tr}(xy)}\\
  =&\sum_{\substack{z\bmod\mathfrak{f}\\(z,\mathfrak{f})=1}}\,
  \sum_{\substack{x\bmod\mathfrak{f}\\(x,\mathfrak{f})=1}}
  \psi_{\mathrm{fin}}(z)e^{2\pi\mathrm{Tr}\left(xy(z-1)\right)}
  =\sum_{\substack{z\bmod\mathfrak{f}\\(z,\mathfrak{f})=1}}\psi_{\mathrm{fin}}(z)
  T_1\left(y(z-1)\right)\\
  =&\sum_{\substack{z\bmod\mathfrak{f}\\(z,\mathfrak{f})=1}}\psi_{\mathrm{fin}}(z)
  \sum_{\mathfrak{a}|\mathfrak{f}}\mu(\mathfrak{a})S_{\mathfrak{a}}\left(y(z-1)\right)
 \end{align*}
Since $(y\mathfrak{fd},\mathfrak{f})=1$, we have
\[y(z-1)\in \mathfrak{a}^{-1}\mathfrak{d}^{-1}\Longleftrightarrow z\equiv1\mod
\frac{\mathfrak{f}}{\mathfrak{a}}.\]
In fact, if
$z-1\in\mathfrak{f}\mathfrak{a}^{-1}$, then 
$y(z-1)\in \mathfrak{a}^{-1}\mathfrak{d}^{-1}$. 
Note that $y\in\mathfrak{f}^{-1}\mathfrak{d}^{-1}$. Conversely,
if $z\not\equiv 1 \bmod\,\frac{\mathfrak{f}}{\mathfrak{a}}$,
since $\frac{\mathfrak{f}}{\mathfrak{a}}\nmid (z-1)$,
there exists a prime factor $\mathfrak{p}$ of $\mathfrak{f}$ 
such that $\nu _{\mathfrak{p}}(z-1)<\nu_{\mathfrak{p}}(y\mathfrak{fd})=0$.
Thus $y(z-1)\mathfrak{ad}\nsubseteq \mathcal{O}_K$ hence
$y(z-1)\notin \mathfrak{a}^{-1}\mathfrak{d}^{-1}$.\\
Therefore
\begin{align*}
  |\tau_{K}(\psi_{\mathrm{fin}},y)|^2
  =&\sum_{\mathfrak{a}|\mathfrak{f}}\mu (\mathfrak{a})
  \sum_{z\bmod\mathfrak{f},\,(z,\mathfrak{f})=1}\psi_{\mathrm{fin}}(z)S_{\mathfrak{a}}
  \left\{y(z-1)\right\}\\
  =&\sum_{\mathfrak{a}|\mathfrak{f}}\mu(\mathfrak{a})
  \sum_{z\bmod\mathfrak{f},\,z\equiv1\bmod\frac{\mathfrak{f}}{\mathfrak{a}}}
  \psi_{\mathrm{fin}}(z)\mathbb{N}_{K/\mathbb{Q}}\left(\frac{\mathfrak{f}}{\mathfrak{a}}\right)
  =\mathbb{N}_{K/\mathbb{Q}}(\mathfrak{f})
\end{align*} 
\end{proof}
By Theorem \ref{absolute value of Gauss1}, 
we obtain the following properties of $\tau_{K}(\psi)$.
\begin{theorem}\label{theorem;Gauss sum}
  Let $\psi$, $\psi_1$, $\psi_2$ be a primitive Hecke character 
    $\mod\mathfrak{f}$,$\mod\mathfrak{f}_{1}$,$\mod\mathfrak{f}_{2}$ respectively.
    Suppose that $(\mathfrak{f}_1,\mathfrak{f}_2)=1$. Then it follows that
   \begin{enumerate}
    \item $\left|\tau_{K}(\psi)\right| ^2=\mathbb{N}_{K/\mathbb{Q}}(\mathfrak{f})$,
    \item $\tau_{K}(\psi_1\psi_2)=\psi_1(\mathfrak{f}_2)
    \psi_2(\mathfrak{f}_1)\tau_{K}(\psi_1)\tau_{K}(\psi_2)$.
  \end{enumerate}
\end{theorem}
\begin{proof}
We apply Theorem \ref{absolute value of Gauss1} to show the  first assertion.
We take a integral ideal $\mathfrak{a}$ of $K$ and $a\in K^\times$
as in Definition \ref{definition of Gauss sum}.
Note that $a^{-1}\in\mathfrak{f}^{-1}\mathfrak{d}^{-1}$.
For $b\in\mathfrak{a}$  with $(b,\mathfrak{f})=1$, the map
$\mathcal{O}_{K}/\mathfrak{f}\ni \overline{x}\to \overline{xb}\in\mathfrak{a}/\mathfrak{af}$
is a bijection. Therefore,
\begin{align*}
  |\tau_{K}(\psi)|=&\left|\sum_{x\bmod \mathfrak{f}}
  \psi_{\mathrm{fin}}(bx)e^{2\pi i\frac{bx}{a}}\right|
  =\left|\sum_{x\bmod \mathfrak{f}}
  \psi_{\mathrm{fin}}(x)e^{2\pi i\frac{bx}{a}}\right|
  =|\tau_{K}(\psi_{\mathrm{fin}},ba^{-1})|
  =\mathbb{N}_{K/\mathbb{Q}}(\mathfrak{f}).
\end{align*}
Secondly assertion is proved by a direct computation.
We take a ideal $\mathfrak{a}$ of $K$ and $a\in K$
such that $\mathfrak{adf_{1}f_{2}}=(a)$ and
$(\mathfrak{a},\mathfrak{f_{1}f_{2}})=1$.
Set $\mathfrak{a}_{1}=\mathfrak{af_{2}}$,
$\mathfrak{a}_{2}=\mathfrak{af_{1}}$.
Note that $(\mathfrak{a}_{i},\mathfrak{f}_{i})=1$ for $i=1,2$. 
We also chose $b_{1},b_{2}\in\mathcal{O}_{K}$ so that
$b_{1}\in \mathfrak{f_{2}},\,(b_{1},\mathfrak{a_{2}})=1$,
$b_{2}\in \mathfrak{f_{1}},\,(b_{2},\mathfrak{a_{1}})=1$.
Under this setting, we calculate as follows:
\begin{align*}
  \tau_{K}(\psi_{1}\psi_{2})=&
  \frac{\psi_{1,\infty}(a)\psi_{2,\infty}(a)}
  {\psi_{1}(\mathfrak{a})\psi_{2}(\mathfrak{a})}
  \sum_{x\in\mathfrak{a}\bmod\mathfrak{af_{1}f_{2}}}
  \psi_{1,\mathrm{fin}}(x)\psi_{2,\mathrm{fin}}(x)
  e^{2\pi i\mathrm{Tr}(\frac{x}{a})}\\
  =& \frac{\psi_{1,\infty}(a)\psi_{2,\infty}(a)}
  {\psi_{1}(\mathfrak{a})\psi_{2}(\mathfrak{a})}
  \sum_{\substack{x_{1}\in\mathfrak{a}\bmod\mathfrak{af_1}\\
  x_{2}\in\mathfrak{a}\bmod\mathfrak{af_{2}}}}
  \psi_{1,\mathrm{fin}}(b_{1}x_{1}+b_{2}x_{2})
  \psi_{2,\mathrm{fin}}(b_{1}x_{1}+b_{2}x_{2})
  e^{2\pi i\mathrm{Tr}(\frac{(b_{1}x_{1}+b_{2}x_{2})}{a})}\\
  =& \frac{\psi_{1,\infty}(a)\psi_{2,\infty}(a)}
  {\psi_{1}(\mathfrak{a})\psi_{2}(\mathfrak{a})}
  \sum_{x_{1}\in\mathfrak{a}\bmod\mathfrak{af_{1}}}
  \psi_{1,\mathrm{fin}}(b_{1}x_{1})e^{2\pi i\mathrm{Tr}(\frac{b_{1}x_{1}}{a})}
  \sum_{x_{2}\in\mathfrak{a}\bmod\mathfrak{af_{2}}}
  \psi_{2,\mathrm{fin}}(b_{2}x_{2})e^{2\pi i\mathrm{Tr}(\frac{b_{2}x_{2}}{a})}\\
  =&\frac{\psi_{1,\infty}(a)\psi_{2,\infty}(a)}
  {\psi_{1}(\mathfrak{a})\psi_{2}(\mathfrak{a})}
  \sum_{x'_{1}\in\mathfrak{a_{1}}\bmod\mathfrak{a_{1}f_1}}
  \psi_{1,\mathrm{fin}}(x'_1)e^{2\pi i\mathrm{Tr}(\frac{x'_{1}}{a})}
  \sum_{x'_{2}\in\mathfrak{a_{2}}\bmod\mathfrak{a_{2}f_2}}
  \psi_{2,\mathrm{fin}}(x'_2)e^{2\pi i\mathrm{Tr}(\frac{x'_{2}}{a})}\\
  =&\frac{\psi_{1,\infty}(a)\psi_{2,\infty}(a)\psi_1(\mathfrak{f_2})\psi_2(\mathfrak{f_1})}
  {\psi_1(\mathfrak{a_1})\psi_2(\mathfrak{a_2})}
  \sum_{x'_{1}\in\mathfrak{a_{1}}\bmod\mathfrak{a_{1}f_1}}
  \psi_{1,\mathrm{fin}}(x'_1)e^{2\pi i\mathrm{Tr}(\frac{x'_{1}}{a})}&\\
   &\hspace{65mm}\times\sum_{x'_{2}\in\mathfrak{a_{2}}\bmod\mathfrak{a_{2}f_2}}
  \psi_{2,\mathrm{fin}}(x'_2)e^{2\pi i\mathrm{Tr}(\frac{x'_{2}}{a})}&\\
  =&\psi_{1}(\mathfrak{f_{2}})\psi_{2}(\mathfrak{f_{1}})
  \tau_{K}(\psi_{1})\tau_{K}(\psi_{2}).
\end{align*}
So we complete the proof.
\end{proof}
Next we study $\psi_\infty$.
\begin{theorem}(\cite[p476-477]{Neukirch1992})\label{theorem;explicit expresion of infinite part}
  Let $\psi_{\infty}$ be a character of $K_{\infty}^\times$, that is, a continuous homomorphism 
  $K_{\infty}^\times\rightarrow S^1$. There exist
  $\displaystyle p\in\prod_{v\in S_{K}}\{0,1\}\times\prod_{v\in S'_{K}}\mathbb{Z}$ and
  $\displaystyle q\in\prod _{v\in S_{K}\cup S'_{K}}\mathbb{R}$ uniquely such that
  $\psi_\infty(x)=\mathbb{N}_{K/\mathbb{Q}}(x^p|x|^{-p+iq})$.
  Here, for $x=(x_1,\cdots,x_r)$, we define $\mathbb{N}_{K/\mathbb{Q}}(x)=x_1\cdots x_r$
  where $r$ is the number of infinite places of $K$.
\end{theorem}
\begin{proof}
  Since $\displaystyle x=\frac{x}{|x|}\cdot |x|$ for $x\in K_{\infty}^{\times}$, we have
  $K_{\infty}^\times =U\times K^\times_{\infty,+}$ where
  $\displaystyle U=\prod_{v\in S_{K}}\{1,-1\}\times \prod_{v\in S'_{K}}S^1$, and
  $\displaystyle K^\times_{\infty,+}=\prod _{v\in S_{K}\cup S'_{K}}\mathbb{R}_{+}$.
  There are only two homomorphisms $\{1,-1\}\rightarrow S^1$, namely the trivial character
  and the sign character.
  Moreover all continuous homomorphism $S^1\rightarrow S^1$
  are $k$ th ($k\in\mathbb{Z}$) power maps.
  Therefore there exists a unique
  $\displaystyle p\in\prod_{v\in S_{K}}\{0,1\}\times\prod_{v\in S'_{K}}\mathbb{Z}$
  such that $\psi_\infty(x) =\mathbb{N}_{K/\mathbb{Q}}(x^p)$
  for any $x\in U$.
  Next we consider characters of $K^\times_{\infty,+}$.
  There is the following isomorphism of  topological groups
\[\log :K^\times_{\infty,+}\rightarrow \mathbb{R}^{r_1+r_2}.\]
Since any continuous homomorphism $\mathbb{R}\rightarrow S^1$ can be written as
$x\mapsto e^{itx}$ for some $t\in\mathbb{R}$,
there exists $\displaystyle q\in\prod_{v\in S_{K}\cup S'_{K}}\mathbb{R}$ uniquely such that 
$\psi_{\infty}(x)=\mathbb{N}_{K/\mathbb{Q}}(x^{iq})$.
for any $x\in K^\times_{\infty,+}$.
Thus, for any $x\in K_{\infty}^\times$,
\begin{align*}
  \psi_{\infty}(x)=&\psi_{\infty}\left(\frac{x}{|x|}\right)\psi_{\infty}(|x|)
                  =\mathbb{N}_{K/\mathbb{Q}}\left(\left(\frac{x}{|x|}\right)^p\right)\mathbb{N}_{K/\mathbb{Q}}(|x|^{iq})
                  =\mathbb{N}_{K/\mathbb{Q}}
                  \left(x^p|x|^{-p+iq}\right)
\end{align*}
\end{proof}
\begin{definition}\label{definition;type of Hecke character}
  Let $\psi =\psi_{\mathrm{fin}}\psi_{\infty}$ be a Hecke character.
  $(p,q)$ determined in 
  Proposition \ref{theorem;explicit expresion of infinite part}
  for $\psi_{\infty}$
  is caled the type of $\psi$.
\end{definition}

\subsection{Hecke $L$-functions}\label{subsec;Lfunction}
In Section \ref{subsec;Lfunction}, we describe some 
analytic properties of Hecke $L$-function.
We restrict ourselves to real quadratic fields since this is the only case we need. 
However, we note that Theorem \ref{theorem;Hecke L function} can be generalized to any number field.  
In what follows, let $F$ be real quadratic field, and let $\psi$ be a Hecke character 
modulo $\mathfrak{f}$ on $F$.
Then there exist a homomorphism 
$\psi_{\mathrm{fin}}\colon \left(\mathcal{O}_{F}/\mathfrak{f}\mathcal{O}_{F}\right)^\times\to S^1$
and a continuous homomorphism 
$\psi_{\infty}\colon F_{\infty}^{\times}\to S^1$ such that 
$\psi((a))=\psi_{\mathrm{fin}}(\overline{a})\psi_{\infty}(a)$
for any $a\in \mathcal{O}_{F}$ prime to $\mathfrak{f}$.  
By Theorem \ref{theorem;explicit expresion of infinite part}, 
the infinite part $\psi_{\infty}$ is given by
$\psi_{\infty}(x)=
(\mathrm{sgn}(x_1))^{\epsilon_1}(\mathrm{sgn}(x_2))^{\epsilon_2}|x_1|^{\nu_1}|x_2|^{\nu_2}$
for some $\epsilon_1, \epsilon_2\in \{0,1\}$ and $\nu_1,\nu_2\in i\mathbb{R}$.
We call $(\epsilon_1, \epsilon_2, \frac{\nu_1}{i},\frac{\nu_2}{i})$ the type of $\psi$
(see Definition \ref{definition;type of Hecke character}). 
We define the Hecke $L$-function associated with $\psi$
by
\[L(s,\psi)=\sum_{\mathfrak{a}}\psi (\mathfrak{a})\mathbb{N}_{F/\mathbb{Q}}(\mathfrak{a})^{-s}\]
if $\mathrm{Re}(s)>1$.
\begin{theorem}(\cite[p.93]{Miyake2006})\label{theorem;Hecke L function}
Let $F$ be a real quadratic field and let $D$ be its discriminant.
Let $\psi$ be a primitive Hecke character modulo $\mathfrak{f}$ of type 
$(\epsilon_1,\epsilon_2,\frac{\nu}{i},-\frac{\nu}{i})$ where
$\epsilon_1,\epsilon_2\in\{0,1\}$, $\nu\in i\mathbb{R}$.
Then
\[\Lambda(s,\psi)=\pi ^{-s}\left(D\mathbb{N}_{F/\mathbb{Q}}(\mathfrak{f})\right)^{\frac{s}{2}}
\Gamma\left(\frac{s+\epsilon_1+\nu}{2}\right)\Gamma\left(\frac{s+\epsilon_2+-\nu}{2}\right)
L(s,\psi)\]
is extended to $\mathbb{C}$ meromorphically.
More precisely, if $\psi$ is not trivial, $\Lambda(s,\psi)$ is holomorphic
on $\mathbb{C}$, and if $\psi$ is trivial, it has poles of order 1 
at $s=1,0$ and holomorphic on $\mathbb{C}\backslash\{0,1\}$.
$\Lambda(s,\psi)$ satisfies the following functional equation:
\[\Lambda(1-s,\psi)=T(\psi)\Lambda(s,\overline{\psi})\]
where
\[T(\psi)=i^{-\epsilon_{1}-\epsilon_{2}}\frac{\tau _F(\psi)}{\mathbb{N}_{F/\mathbb{Q}}(\mathfrak{f})^\frac{1}{2}}.\]
Moreover, if $\sigma_{1}<\sigma_{2}$ are arbitrary real numbers,
and $t_1$ is a positive real number,
$\Lambda(s,\psi)$ is uniformly bounded for 
$\sigma_{1}<\mathrm{Re}(s)<\sigma_{2}$, $|\mathrm{Im}(s)|>t_1$.
\end{theorem}
\begin{remark}\label{remark;Pharagmen-Lindelöf}
  By the Phragmén-Lindelöf theorem, for  any $\sigma_1< \sigma_2$,
 $\Lambda(s,\psi)$ decays rapidly and uniformly for
$\sigma_1<\mathrm{Im}(s)<\sigma_2$
as $|\mathrm{Im}(s)|\rightarrow\infty$. 
\end{remark}
For an application of this theorem, 
we can show the following theorem. 
\begin{lemma}(\cite[p110]{Bump1997})\label{lemma;relation in Gauss sums}
  Let $\chi_D$ be the quadratic character associated with $F$.
  Let $p\nmid D$. For a primitive Dirichlet character 
  $\sigma \colon\left(\mathbb{Z}/p\mathbb{Z}\right)^\times
  \rightarrow S^1$, the following relation holds:
  \[\tau_{F}(\sigma \circ \mathbb{N}_{F/\mathbb{Q}})=D^{-\frac{1}{2}}
  \tau _\mathbb{Q}(\sigma)\sigma_{\mathbb{Q}}(\sigma\chi_D)=
  \sigma(D)\chi_{D}(p)\tau_\mathbb{Q}(\sigma)^2\]
  where $\sigma(-1)=(-1)^\delta$\quad($\delta\in \{0,1\}$).
\end{lemma}
We take an analytic approach here. 
However, in \cite{Bump1997}, it is remarked that
Lemma \ref{lemma;relation in Gauss sums}
can be proved by an algebraic method.
\begin{proof}
  We denote  
  by $L(s,\sigma),L(s,\sigma\chi_{D})$
  the Dirichlet $L$-function 
  associated with Dirichlet characters $\sigma$,  $\sigma\chi_{D}$ respectively.
  Set
  \begin{align*}
    \Lambda(s,\sigma)=&\pi^{-\frac{s+\delta}{2}}
    \Gamma\left(\frac{s+\delta}{2}\right)
    L(s,\sigma)\\
    \Lambda(s,\sigma\chi_{D})=&\pi^{-\frac{s+\delta}{2}}
    \Gamma\left(\frac{s+\delta}{2}\right)
    L(s,\sigma\chi_D)\\
    \Lambda(s,\sigma\circ\mathbb{N}_{F/\mathbb{Q}})=&\pi^{-s}D^{\frac{s}{2}}
    p^s\Gamma\left(\frac{s+\delta}{2}\right)L(s,\sigma\circ\mathbb{N}_{F/\mathbb{Q}}).
  \end{align*}
The functional equations of
  $\Lambda(s,\sigma)$, $\Lambda(s,\sigma\chi_D)$ and $\Lambda(s,\sigma\circ\mathbb{N}_{F/\mathbb{Q}})$
  are
  \begin{empheq}[left=\empheqlbrace]{align}
      \Lambda(1-s,\sigma)&=\frac{\tau_{\mathbb{Q}(\sigma)}}{i^\delta p^{\frac{1}{2}}}
      \Lambda(s,\overline{\sigma})\label{ft1}\\
      \Lambda(1-s,\sigma\chi_D)&=\frac{\tau_{\mathbb{Q}}(\sigma\chi_D)}{i^\delta(Dp)^\frac{1}{2}}
      \Lambda(s,\overline{\sigma}\chi_{D})\label{ft2}\\
      \Lambda(1-s,\sigma\circ\mathbb{N}_{F/\mathbb{Q}})&=T(\sigma\circ\mathbb{N}_{F/\mathbb{Q}})
      \Lambda(s,\overline{\sigma}\circ\mathbb{N}_{F/\mathbb{Q}})\label{ft3} .
  \end{empheq}
  Moreover,
   \begin{align*}
    L(s,\sigma\circ\mathbb{N}_{F/\mathbb{Q}})
    =&\prod _{\mathfrak{p}}
    \left(1-\sigma(\mathbb{N}_{F/\mathbb{Q}}(\mathfrak{p}))\mathbb{N}_{F/\mathbb{Q}}(\mathfrak{p})^{-s}\right)^{-1}\\
    =&\prod_{\chi_{D}(p)=1}\left(1-\sigma(p)p^{-s}\right)^{-2}
    \prod_{\chi_{D}(p)=0}\left(1-\sigma(p)p^{-s}\right)^{-1}
    \prod_{\chi_{D}(p)=-1}\left(1-\sigma(p^2)p^{-2s}\right)^{-1}\\
    =&\prod_{p}\left(1-\sigma(p)p^{-s}\right)^{-1}
    \prod_{\chi_{D}(p)=1}\left(1-\sigma(p)p^{-s}\right)^{-1}
    \prod_{\chi_{D}(p)=-1}\left(1+\sigma(p)p^{-s}\right)^{-1}\\
    =&L(s,\sigma)L(s,\chi_{D}\sigma).
  \end{align*}
  Therefore, we have
  $\displaystyle \Lambda(s,\sigma\circ\mathbb{N}_{F/\mathbb{Q}})
  =\pi^\delta\Lambda(s,\sigma)\Lambda(s,\chi_{D}\sigma)$.
  Applying (\ref{ft1}), (\ref{ft2}),
  \begin{align*}
    \Lambda(1-s,\sigma\circ\mathbb{N}_{F/\mathbb{Q}})=&\pi^{\delta}
    \Lambda(1-s,\sigma)\Lambda(1-s,\sigma\chi_D)
    =\pi^\delta\frac{\tau_\mathbb{Q}(\sigma)}{i^\delta p^\frac{1}{2}}
    \Lambda(s,\overline{\sigma})\cdot\frac{\tau_{\mathbb{Q}}(\sigma\chi_D)}{i^\delta(Dp)^\frac{1}{2}}
    \Lambda(s,\overline{\sigma}\chi_D)\\
    =&(-1)^\delta\frac{\tau_{\mathbb{Q}}(\sigma)\tau_{\mathbb{Q}}(\sigma\chi_D)}{pD^\frac{1}{2}}
    \Lambda(s,\overline{\sigma}\circ\mathbb{N}_{F/\mathbb{Q}}).
  \end{align*}
 Comparing (\ref{ft3}), we obtain 
 \[ T(\sigma\circ\mathbb{N}_{F/\mathbb{Q}})=(-1)^\delta\frac{\tau_{\mathbb{Q}}(\sigma)\tau_{\mathbb{Q}}(\sigma\chi_{D})}{pD^\frac{1}{2}}.\]
 Since $\displaystyle T(\sigma\circ\mathbb{N}_{F/\mathbb{Q}})=(-1)^\delta\frac{\tau_{F}(\sigma\circ\mathbb{N}_{F/\mathbb{Q}})}{p}$,
 $\tau_{F}(\sigma\circ\mathbb{N}_{F/\mathbb{Q}})=D^{-\frac{1}{2}}\tau_{\mathbb{Q}}(\sigma)\tau_{\mathbb{Q}}(\sigma\chi_{D})$.
Another equality is clear. 
\end{proof}

\section{Construction of Maass wave forms}\label{sec;main result1}
Based on the previous section, we construct Maass wave forms associated with 
Hecke character on real quadratic field $F$.
Let $\psi$ be a primitive Hecke character modulo $\mathfrak{f}$ of the type
$(\epsilon,\epsilon,\frac{\nu}{i},\frac{-\nu}{i})$ where
$\epsilon\in\{0,1\},\nu\in i\mathbb{R}$.
Namely, its infinite part $\psi_{\infty}$ is given by
$\psi_{\infty}(x)=\left(\frac{x_1}{|x_1|}\right)^{\epsilon}
\left(\frac{x_2}{|x_2|}\right)^{\epsilon}|x_1|^{\nu}|x_2|^{-\nu}$. 
As in Section \ref{sec;Introduction}, we define the function
$\Theta_{\psi}$ on $\mathbb{H}=\{z=x+iy\mid x\in\mathbb{R},y>0\}$ by
\begin{equation*}
		\Theta_{\psi}(z)=
		\begin{cases}
			\sum_{\mathfrak{a}}\psi(\mathfrak{a})\sqrt{y}K_{\nu}(2\pi\mathbb{N}_{F/\mathbb{Q}}(\mathfrak{a})y)
			\cos (2\pi\mathbb{N}_{F/\mathbb{Q}}(\mathfrak{a})x)\quad&(\text{if }\epsilon =0),\\
			\sum_{\mathfrak{a}}\psi(\mathfrak{a})\sqrt{y}K_{\nu}(2\pi\mathbb{N}_{F/\mathbb{Q}}(\mathfrak{a})y)
			\sin (2\pi\mathbb{N}_{F/\mathbb{Q}}(\mathfrak{a})x)&(\text{if }\epsilon =1)
		\end{cases}
\end{equation*}
where $\mathfrak{a}$ runs over all integral ideals of $F$ and
$\displaystyle K_{\nu}(y)=\frac{1}{2}\int_{0}^{\infty}e^{-y(t+t^{-1})/2}\,t^\nu\frac{dt}{t}$
is the $K$-Bessel function. 
If $y>1$, then
\[\displaystyle |K_{\nu}(y)|\leq\int_{1}^{\infty}e^{-y(t+t^{-1})}\frac{dt}{t}=
\int_{y}^{\infty}e^{-t}\frac{dt}{t}\leq\int_{y}^{\infty}e^{-t}dt
=e^{-y}.\]
Thus, $K_{\nu}(y)$ decays rapidly as $y\rightarrow \infty$.
Note that $K_{\nu}(y)$ is real analytic.\\
For a function $\Theta$ on $\mathbb{H}$ and
$\begin{pmatrix}
  a&b\\c&d
\end{pmatrix}\in GL_{2}(\mathbb{R})^{+}$, we define the slash operator by 
\[\Theta\left|\begin{pmatrix}
  a&b\\c&d
\end{pmatrix}\right. (z)=
\Theta\left(\frac{az+b}{cz+d}\right).\]
Our goal is to prove Theorem \ref{main2},
which contains Theorem \ref{main1} as a special case.
To this end, we prepare the following lemma.
\begin{lemma}(\cite[p109]{Bump1997})\label{bump}
  Let $f$ be a smooth function on $\mathbb{H}$ which is an eigenfunction of 
  $\displaystyle\Delta=-y^{2}\left(\frac{\partial ^2}{\partial x^2}+\frac{\partial ^2}{\partial y^2}\right)$.
  Suppose that the function $f(iy)$ is real analytic for $y>0$ and satisfies
  $\displaystyle f(iy)=\frac{\partial f}{\partial x}(iy)=0$,
  then $f=0$.
\end{lemma}
When we consider a holomorphic modular form, we can use the identity theorem.
However, in our case, we consider non-holomorphic functions. Thus, we apply 
Lemma \ref{bump} instead of the identity theorem.
\begin{proof}
  It is suffice to show that $\displaystyle \frac{\partial ^n f}{\partial x^n}$ vanishes on 
  the imaginary axis for any $n$.
  If it is true, we have 
  \[\frac{\partial ^{n+m}f}{\partial x^n\partial y^m}(iy)
  =\frac{\partial ^m}{\partial y^m}0=0.\]
  Therefore, the analyticity of $f$ and the connectedness of 
  $\mathbb{H}$ implies $f=0$ on $\mathbb{H}$.\\
  By the assumption, for $n=0,1$, we have
  $\displaystyle \frac{\partial ^n f}{\partial x^n}(iy)=0$.
  If $n\geq 2$, taking eigenvalue of $f$; $\Delta f=\lambda f$,
  \begin{align*}
    \frac{\partial ^nf}{\partial x^n}(iy)
    =&\frac{\partial^{n-2}}{\partial x^{n-2}}
    \left(-\frac{1}{y^2}\Delta f(iy)-\frac{\partial ^2f}{\partial y^2}(iy)\right)
    =-\frac{\lambda}{y^2}\frac{\partial ^{n-2}f}{\partial x^{n-2}}(iy)
    =0.
  \end{align*}
  Thus, we obtain the conclusion.
\end{proof}

 In Section \ref{sec;main result1}, 
 we show more general statement Theorem \ref{main2}.
  \begin{theorem}\label{main2}
  Let $\psi\ne 1$ be a primitive Hecke character. 
  Then $\Theta_{\psi}$ defined above holds the following
  properties.:
  \begin{enumerate}
    \item For any 
    $\begin{pmatrix}
    a&b\\
    c&d
    \end{pmatrix}\in
    \Gamma_0(D\mathbb{N}_{F/\mathbb{Q}}(\mathfrak{f}))$, it follows that
    $\displaystyle\Theta_\psi \left|\begin{pmatrix}
      a&b\\
      c&d
    \end{pmatrix}\right.
    =\chi_{D}(d)\psi_{\mathrm{fin}}(d)\Theta_{\psi}$;
    \item $\Theta_\psi$ is an eigenfunction of 
    $\displaystyle\Delta=-y^2\left(\frac{\partial ^2}{\partial x^2}+\frac{\partial ^2}{\partial y^2}\right)$
    with eigenvalue $\frac{1}{4}-\nu^2$;
    \item $\Theta$ is of moderate growth at cusps.
  \end{enumerate}
  In other words, 
  $\Theta_{\psi}\in M(\Gamma_{0}(D\mathbb{N}_{F/\mathbb{Q}}(\mathfrak{f})),\nu, \psi_{\mathit{fin}}\chi_{D})$.
  Furthermore, if $\psi$ is not of the form $\chi\circ\mathbb{N}_{F/\mathbb{Q}}$,
  then $\Theta_{\psi}$ decays rapidly at cusps, that is, 
  $\Theta_{\psi}\in S(\Gamma_{0}(D\mathbb{N}_{F/\mathbb{Q}}(\mathfrak{f})),\nu, \psi_{\mathit{fin}}\chi_{D})$.
\end{theorem}

The proof employs the method of the Weil converse theorem and analytic properties
Hecke\,$L$-function. 
The proof is essentially same in each case $\epsilon=0$ and
$\epsilon=1$. However, we treat the case $\epsilon =0$ and $\epsilon=1$ 
separately, because 
there are technical differences in the computation depending on parity.

\subsection{Proof for the case $\epsilon=0$.}
Firstly, we show in the case of $\epsilon =0$. 
We need the following lemma.
\begin{lemma}
  Let $p$ be an inert prime with
  $p \nmid\mathbb{N}_{F/\mathbb{Q}}(\mathfrak{f})$ and 
  let $\sigma \colon (\mathbb{Z}/p\mathbb{Z})^\times \to S^1$ be a map
  (not necessarily a character). Define a function $f_{\sigma,\psi}$ 
  on $\mathbb{H}$ by
  \begin{equation}\label{f}
  f_{\sigma,\psi}=\sum_{\substack{m \bmod p\\(m,p)=1}}\overline{\sigma}(m) 
  \Theta_\psi\left|\begin{pmatrix}
  p&m\\ &p
  \end{pmatrix}\right. .
  \end{equation}
  Then it follows that 
  \begin{enumerate}
    \item \begin{equation}\label{eq;automorphy of theta0}
          \Theta_\psi(y) = T(\psi)\, \Theta_{\overline{\psi}}\!\left|
          \begin{pmatrix}
                               &-1\\
          D\mathbb{N}_{F/\mathbb{Q}}(\mathfrak{f})&
          \end{pmatrix}
          \right. ,
          \end{equation}
    \item \begin{equation}\label{f;automorphy1}
           f_{\sigma,\psi}\left|\begin{pmatrix}
                              &-1\\
           D\mathbb{N}_{F/\mathbb{Q}}(p\mathfrak{f})&
           \end{pmatrix}\right.
           = -T(\psi)\psi_{\mathrm{fin}}(p)f_{\rho,\overline{\psi}}
          \end{equation}
          where $\rho\colon (\mathbb{Z}/p\mathbb{Z})^\times \to S^1$
          is defined by 
          $\rho(m) \coloneq \sigma(r)$
          with $-rmD \mathbb{N}_{F/\mathbb{Q}}(\mathfrak{f}) 
           \equiv 1 \bmod p$.
  \end{enumerate}
\end{lemma}
\begin{proof}
If $\mathrm{Re}(s)>1$,
\begin{align*}
\int_0^{\infty} \Theta_\psi(iy) y^{s-\frac{1}{2}} \,\frac{dy}{y}
  =&\sum_{\mathfrak{a}}\psi(\mathfrak{a})
    \int_0^{\infty} K_\nu\left(2\pi \mathbb{N}_{F/\mathbb{Q}}(\mathfrak{a})y\right) \, y^s \, \frac{dy}{y}\\
  =& \sum_{\mathfrak{a}}\psi(\mathfrak{a})(2\pi \mathbb{N}_{F/\mathbb{Q}}(\mathfrak{a}))^s
    \int_0^\infty K_{\nu} (y)y^s\,\frac{dy}{y}\\
  =&\sum_{\mathfrak{a}}\psi(\mathfrak{a})\left(2\pi\mathbb{N}_{F/\mathbb{Q}}(\mathfrak{a})\right)^{-s}
  2^{s-2}\Gamma\left(\frac{s-\nu}{2}\right)\Gamma\left(\frac{s+\nu}{2}\right)\\
  =&\frac{1}{4}\pi^{-s}\sum_{\mathfrak{a}}\psi(\mathfrak{a})\mathbb{N}_{F/\mathbb{Q}}(\mathfrak{a})^{-s}
  \Gamma\left(\frac{s-\nu}{2}\right)\Gamma\left(\frac{s+\nu}{2}\right)\\
  =&\frac{1}{4}\pi^{-s}L(s,\psi)\Gamma\left(\frac{s-\nu}{2}\right)
  \Gamma\left(\frac{s+\nu}{2}\right)\\
  =& \frac{1}{4}(D\mathbb{N}_{F/\mathbb{Q}}(\mathfrak{f}))^{-s/2} \Lambda(s, \psi).
\end{align*}
Applying the Mellin inversion formula, 
\begin{equation}\label{the Mellin inversion formula}
  \Theta_\psi(iy)
  =\frac{\sqrt{y}}{8\pi i}
    \int_{R - i\infty}^{R + i\infty}
      (D\mathbb{N}_{F/\mathbb{Q}}(\mathfrak{f}))^{-s/2} \Lambda(s, \psi) \, y^{-s} \, ds.
\end{equation}
Here $R$ is a real number with $R>1$. 
By Remark \ref{remark;Pharagmen-Lindelöf}
and the Cauchy integral theorem,
(\ref{the Mellin inversion formula}) holds for
any $R\in\mathbb{R}$ only when $\Lambda(s,\psi)$ is entire. Applying the functional
equation of $\Lambda(s,\psi)$, we compute as follows: 
\begin{align*}
  \Theta_{\psi}(iy)
  =&\frac{\sqrt{y}}{8\pi i}T(\psi)
    \int_{R - i\infty}^{R + i\infty}
     (D\mathbb{N}_{F/\mathbb{Q}}(\mathfrak{f}))^{-s/2} \Lambda(1-s, \overline{\psi}) \, y^{-s} \, ds\\
  =&\frac{\sqrt{y}}{8\pi i}T(\psi)
    \int_{R - i\infty}^{R + i\infty}
     (D\mathbb{N}_{F/\mathbb{Q}}(\mathfrak{f}))^{-\frac{1-s}{2}} \Lambda(s, \overline{\psi}) \, y^{1-s} \, ds\\
  =&\frac{1}{8\pi i \sqrt{y D \mathbb{N}_{F/\mathbb{Q}} (\mathfrak{f})}}
     \, T(\psi)
     \int_{R - i\infty}^{R + i\infty}
       (D \mathbb{N}_{F/\mathbb{Q}} (\mathfrak{f}))^{-s/2}
       \Lambda(s, \overline{\psi})
       (D \mathbb{N}_{F/\mathbb{Q}} (\mathfrak{f})y)^{s} \, ds \\[6pt]
  =&T(\psi)\, \Theta_{\overline\psi} \!\left((D\mathbb{N}_{F/\mathbb{Q}}(\mathfrak{f}))^{-1} y \right).
\end{align*}
By Lemma \ref{bump}, we have Equation \eqref{eq;automorphy of theta0}.\\
Since the linear space that consists of 
all maps \((\mathbb{Z}/p\mathbb{Z})^\times\to S^1\)
is spanned by Dirichlet characters,
we may assume that $\sigma$ is a Dirichlet character 
to show Equation (\ref{f;automorphy1}). 
In this case, we have $\rho(m) = \overline{\sigma}\left(-D \mathbb{N}_{F/\mathbb{Q}}(\mathfrak{f})\right)
\overline{\sigma}(m)$, hence Equation \eqref{f;automorphy1} is equivalent to
the following equation:
\begin{equation} \label{f;automorphy2}
 f_{\sigma,\psi}\left|\begin{pmatrix}
                            &-1\\
  D\mathbb{N}_{F/\mathbb{Q}}(p\mathfrak{f})&
  \end{pmatrix}\right.
  =-T(\psi)\psi_{\mathrm{fin}}(p)\sigma \left(-D\mathbb{N}_{F/\mathbb{Q}} (\mathfrak{f})\right)
     f_{\overline{\sigma},\overline{\psi}}.
\end{equation}
Now we show Equation \eqref{f;automorphy2}.
Let $a(n)$ be the Fourier coefficients of $\Theta_{\psi}$:
\begin{equation}\label{fourier coefficient}
  \Theta_\psi(z)
  =\sum_{n \ne 0} a(n)\sqrt{y} K_{\nu}(2\pi |n| y)e^{2\pi i n x}.
\end{equation}
We divide three cases. \\
Case 1. The case where \(\sigma\) is primitive and \(\sigma(-1) = 1\).\\
By direct computation,
\begin{align*}
 f_{\sigma,\psi}(z)
   &=\sum_{\substack{m \bmod p,\\ p\nmid m}}
     \overline{\sigma}(m)
     \sum_{n \ne 0}
       a(n)\sqrt{y}\, K_{\nu}(2\pi |n| y)
       e^{2\pi i n x}
       e^{2\pi i \frac{nm}{p}} \\[6pt]
  &= \tau_{\mathbb{Q}}(\overline{\sigma})
     \sum_{n \ne 0}
       a(n)\sigma(n)\sqrt{y} K_{\nu}(2\pi |n| y)
       e^{2\pi inx} \\[6pt]
  &= \tau_{\mathbb{Q}}(\overline{\sigma})
     \sum_{\mathfrak{a}}
       \sigma(\mathbb{N}_{F/\mathbb{Q}}(\mathfrak{a}))\psi(\mathfrak{a})\sqrt{y}
       K_{\nu}(2\pi \mathbb{N}_{F/\mathbb{Q}}(\mathfrak{a})y)
       \cos(2\pi  \mathbb{N}_{F/\mathbb{Q}}(\mathfrak{a})x).
\end{align*}
Therefore the Mellin transform of $f_{\sigma, \psi}$ can be computed as follows:
\begin{align*}
  \int_0^\infty f_{\sigma,\psi}(iy)\, y^{s-\frac{1}{2}}\, \frac{dy}{y}
  =& \tau_{\mathbb{Q}}(\overline{\sigma}) \sum_{\mathfrak{a}}
    \sigma\circ \mathbb{N}_{F/\mathbb{Q}}(\mathfrak{a})\, \psi(\mathfrak{a})
    \int_0^\infty K_{\nu}(2\pi \mathbb{N}_{F/\mathbb{Q}}(\mathfrak{a}) y)\, y^{s}\, \frac{dy}{y}\\
  =&\tau_{\mathbb{Q}}(\overline{\sigma})\sum_{\mathfrak{a}}\sigma\circ\mathbb{N}_{F/\mathbb{Q}}(\mathfrak{a})
  \psi(\mathfrak{a})(2\pi\mathbb{N}_{F/\mathbb{Q}}(\mathfrak{a}))^{-s}2^{s-2}\Gamma\left(\tfrac{s-\nu}{2}\right)
  \Gamma\left(\tfrac{s+\nu}{2}\right)\\
  =& \tfrac{1}{4}\, \tau_{\mathbb{Q}}(\overline{\sigma})\, \pi^{-s}\,
   L(s, (\sigma\circ\mathbb{N}_{F/\mathbb{Q}}) \psi)\,
   \Gamma\!\left(\tfrac{s-\nu}{2}\right)
   \Gamma\!\left(\tfrac{s+\nu}{2}\right) \\[4pt]
  =& \tfrac{1}{4}\, \tau_{\mathbb{Q}}(\overline{\sigma})\,
   (D \mathbb{N}_{F/\mathbb{Q}}(\mathfrak{f}))^{-s/2}\,
   \Lambda(s, (\sigma\circ\mathbb{N}_{F/\mathbb{Q}}) \psi).
\end{align*}
Moreover, by
Theorem \ref{theorem;Gauss sum} and
Lemma \ref{lemma;relation in Gauss sums},
The functional equation of $\Lambda(s,(\sigma\circ\mathbb{N}_{F/\mathbb{Q}})\psi)$
is computed as follows:
\begin{align}
\Lambda(s, (\sigma\circ\mathbb{N}_{F/\mathbb{Q}}) \psi)
  &= T((\sigma\circ\mathbb{N}_{F/\mathbb{Q}}) \psi)\Lambda(1-s,(\overline{\sigma}\circ\mathbb{N}_{F/\mathbb{Q}}) \overline{\psi}) 
  =\frac{\tau_{F}((\sigma\circ\mathbb{N}_{F/\mathbb{Q}}) \psi)}{\mathbb{N}_{F/\mathbb{Q}}(p\mathbb{\mathfrak{f}})^{\frac{1}{2}}}
     \Lambda(1-s, (\overline{\sigma}\circ\mathbb{N}_{F/\mathbb{Q}})\overline{\psi}) \notag\\[4pt]
  &= \frac{\psi_{\mathrm{fin}}(p)\, \sigma\circ \mathbb{N}_{F/\mathbb{Q}}(\mathfrak{f})\, 
  \tau_{F}(\sigma\circ\mathbb{N}_{F/\mathbb{Q}})}{p\mathbb{N}_{F/\mathbb{Q}}(\mathfrak{f})^{1/2}}
  \Lambda(1-s, (\overline{\sigma}\circ\mathbb{N}_{F/\mathbb{Q}}) \overline{\psi})\notag\\[5pt]
  &=\frac{\psi_{\mathrm{fin}}(p)\,\sigma\circ\mathbb{N}_{F/\mathbb{Q}}(\mathfrak{f})}{p}\tau_{F}(\sigma\circ\mathbb{N}_{F/\mathbb{Q}})
  \Lambda(1-s,(\overline{\sigma}\circ\mathbb{N}_{F/\mathbb{Q}})\,\overline{\psi})\notag\\[5pt]
  &= \frac{\psi_{\mathrm{fin}}(p)\, \sigma\circ\mathbb{N}_{F/\mathbb{Q}}((\mathfrak{f}))}{p}
    \sigma(D)\, \chi_D(p)\, \tau_{\mathbb{Q}}(\sigma)^2 \, T(\psi)\,
    \Lambda(1-s, (\overline{\sigma}\circ\mathbb{N}_{F/\mathbb{Q}}) \overline{\psi})\notag\\[5pt]
  &=-\, \frac{\psi_{\mathrm{fin}}(p)\, \sigma\circ\mathbb{N}_{F/\mathbb{Q}}(\mathfrak{f})}{p}\,
     \sigma(D)\, \tau_{\mathbb{Q}}(\sigma)^2\, T(\psi)\,
     \Lambda(1-s, (\overline{\sigma}\circ\mathbb{N}_{F/\mathbb{Q}})\overline{\psi})\label{functional equation1}.
\end{align}

Therefore, by
the Mellin inversion formula and
Remark \ref{remark;Pharagmen-Lindelöf}, we have
\begin{align*}
f_{\sigma,\psi}(iy)
  &= \frac{\sqrt{y}}{8\pi i}\,
    \tau_{\mathbb{Q}}(\overline{\sigma})\,
    \int_{i\infty}^{-i\infty}
      (D \mathbb{N}_{F/\mathbb{Q}}(p\mathfrak{f}))^{-s/2}
      \Lambda(s, (\sigma\circ\mathbb{N}_{F/\mathbb{Q}})\psi)\,
      y^{-s}\, ds\\
   &= \frac{\sqrt{y}}{8\pi i}\,
   \tau_{\mathbb{Q}}(\overline{\sigma})
   \left(
     -\frac{\psi_{\mathrm{fin}}(p)\, \sigma(\mathbb{N}_{F/\mathbb{Q}}(\mathfrak{f}))}{p}\,
       \sigma(D)\, \tau_{\mathbb{Q}}(\sigma)^2\, T(\psi)
   \right)\\
   &\hspace{50mm}\times\int_{-i\infty}^{i\infty}
     (D \mathbb{N}_{F/\mathbb{Q}}(p\mathfrak{f}))^{-s/2}\,
     \Lambda(1-s, (\overline{\sigma}\circ\mathbb{N}_{F/\mathbb{Q}})  \overline{\psi})\,
     y^{-s}\, ds\\
   &= -\, \frac{1}{8\pi i}\frac{\psi_{\mathrm{fin}}(p)}{\sqrt{yD\mathbb{N}_{F/\mathbb{Q}}(p\mathfrak{f})}}
     \sigma(\mathbb{N}_{F/\mathbb{Q}}(\mathfrak{f}))\sigma(D)\tau_{\mathbb{Q}}(\sigma)T(\psi)\\
   &\hspace{50mm}\times\int_{-i\infty}^{i\infty}
       (D \mathbb{N}_{F/\mathbb{Q}}(p\mathfrak{f}))^{-\frac{s}{2}}\,
       \Lambda(s, (\overline{\sigma}\circ\mathbb{N}_{F/\mathbb{Q}}) \overline{\psi})\,
       (D\mathbb{N}_{F/\mathbb{Q}}(p\mathfrak{f})y)^{s}\, ds\\
   &= -\, \psi_{\mathrm{fin}}(p)\,\sigma(-D\mathbb{N}_{F/\mathbb{Q}}(\mathfrak{f}))
    T(\psi)\, f_{\overline{\sigma},\overline{\psi}}\left|
      \begin{pmatrix}
                                  &-1\\
        D\mathbb{N}_{F/\mathbb{Q}}(p\mathfrak{f})&
      \end{pmatrix}\right.
      (iy).
\end{align*}

Applying Lemma \ref{bump}, we obtain Equation (\ref{f;automorphy2}):
\[
f_{\sigma,\psi}(z)
  = -\, \psi_{\mathrm{fin}}(p)\, \sigma(-D \mathbb{N}_{F/\mathbb{Q}}(\mathfrak{f}))\,
  T(\psi)\, f_{\overline{\sigma},\overline{\psi}}\left|
  \begin{pmatrix}
                               &-1\\
    D \mathbb{N}_{F/\mathbb{Q}}(p\mathfrak{f})&
  \end{pmatrix}\right. .
\]

Case 2. The case where $\sigma\colon$ primitive and $\sigma(-1)=-1$.\\
By direct computation, 
\begin{align*}
  f_{\sigma,\psi}(z)
  =& \tau_{\mathbb{Q}}(\overline{\sigma})\,
    \sum_{n \ne 0}
      \sigma(n)\, a(n)\, \sqrt{y}\,
      K_{\nu}(2\pi |n| y)\, e^{-2\pi i n x}\\
  =&i\tau_{\mathbb{Q}}(\overline{\sigma})\sum_{\mathfrak{a}}
  \sigma(\mathbb{N}_{F/\mathbb{Q}}(\mathfrak{a}))\psi(\mathfrak{a})\sqrt{y}
  K_{\nu}(2\pi\mathbb{N}_{F/\mathbb{Q}}(\mathfrak{a})y)\sin (2\pi\mathbb{N}_{F/\mathbb{Q}}(\mathfrak{a})x)
\end{align*}
Therefore,
\[
\frac{\partial f_{\sigma,\psi}}{\partial x}(z)
  = i\, \tau_{\mathbb{Q}}(\overline{\sigma})\,
    \sum_{\mathfrak{a}}
      \sigma(\mathbb{N}_{F/\mathbb{Q}}(\mathfrak{a}))\, \psi(\mathfrak{a})\,
      \sqrt{y}\, K_{\nu}(2\pi \mathbb{N}_{F/\mathbb{Q}}(\mathfrak{a}) y)\,
      (2\pi \mathbb{N}_{F/\mathbb{Q}}(\mathfrak{a})y)\,
      \sin(2\pi \mathbb{N}_{F/\mathbb{Q}}(\mathfrak{a}) x).
\]
Thus,
\begin{align*}
  \int^\infty_0 \frac{\partial f_{\sigma, \psi}}{\partial x}(iy)y^{s+\frac{1}{2}}\frac{dy}{y}
  =&i\tau_{\mathbb{Q}}(\overline{\sigma})\sum_{\mathfrak{a}}
  \sigma\circ\mathbb{N}_{F/\mathbb{Q}}(\mathfrak{a})\psi(\mathfrak{a})(2\pi\mathbb{N}_{F/\mathbb{Q}}(\mathfrak{a}))
  \int^\infty_0K_\nu(2\pi\mathbb{N}_{F/\mathbb{Q}}(\mathfrak{a})y)y^{s+1}\frac{dy}{y}\\
  =&i\tau_{\mathbb{Q}}(\overline{\sigma})\sum_{\mathfrak{a}}
  \sigma\circ\mathbb{N}_{F/\mathbb{Q}}(\mathfrak{a})\psi(\mathfrak{a})(2\pi\mathbb{N}_{F/\mathbb{Q}}(\mathfrak{a}))^{-s}
  2^{s-1}\Gamma\left(\frac{s+1-\nu}{2}\right)
  \Gamma\left(\frac{s+1+\nu}{2}\right)\\
  =&\frac{1}{2}i\tau_{\mathbb{Q}}(\overline{\sigma})\pi^{-s}
  L(s,(\sigma\circ\mathbb{N}_{F/\mathbb{Q}})\psi)
  \Gamma\left(\frac{s+1-\nu}{2}\right)
  \Gamma\left(\frac{s+1-\nu}{2}\right)\\
  =&\frac{1}{2}i\tau_{\mathbb{Q}}(\overline{\sigma})
  (D\mathbb{N}_{F/\mathbb{Q}}(p\mathfrak{f}))^{-s/2}
  \Lambda(s,(\sigma\circ\mathbb{N}_{F/\mathbb{Q}})\psi).
\end{align*}  
By the Mellin inversion formula and 
Remark \ref{remark;Pharagmen-Lindelöf}, we have
\[\frac{\partial f_{\sigma,\psi}}{\partial x}(iy)
  =\frac{1}{4\pi}\frac{1}{\sqrt{y}}\tau_{\mathbb{Q}}(\sigma)
  \int^{i\infty}_{-i\infty}(D\mathbb{N}_{F/\mathbb{Q}}(p\mathfrak{f}))^{-s/2}
  \Lambda(s,(\sigma\circ\mathbb{N}_{F/\mathbb{Q}})\psi)y^{-s}ds.
\]
By Theorem \ref{theorem;Hecke L function} and
Equation \eqref{functional equation1}, we compute as follows:
\begin{align*}
  \frac{\partial f_{\sigma,\psi}}{\partial x}(iy)
  =&\frac{1}{4\pi}\frac{1}{\sqrt{y}}\tau_{\mathbb{Q}}(\overline{\sigma})
  \left(-\frac{\tau_{F}((\sigma\circ\mathbb{N}_{F/\mathbb{Q}})\psi)}{p(\mathbb{N}_{F/\mathbb{Q}}(\mathfrak{f}))}\right)
  \int^{i\infty}_{-i\infty}(D\mathbb{N}_{F/\mathbb{Q}}(p\mathfrak{f}))^{-s/2}
  \Lambda(1-s,(\overline{\sigma}\circ\mathbb{N}_{F/\mathbb{Q}})\overline{\psi})y^{-s}ds\\[5pt]
  =&-\frac{1}{4\pi}\frac{1}{\sqrt{y}}\tau_{\mathbb{Q}}(\overline{\sigma})
  \frac{\psi_{\mathrm{fin}}(p)\sigma\circ\mathbb{N}_{F/\mathbb{Q}}(\mathfrak{f})\tau_{F}(\psi)\tau_{F}(\sigma\circ\mathbb{N}_{F/\mathbb{Q}})}{p(\mathbb{N}_{F/\mathbb{Q}}(\mathfrak{f}))^{-1/2}}\\
  &\hspace{50mm}\times\int^{i\infty}_{-i\infty}(D\mathbb{N}_{F/\mathbb{Q}}(p\mathfrak{f}))^{-s/2}
  \Lambda(1-s,(\overline{\sigma}\circ\mathbb{N}_{F/\mathbb{Q}})\overline{\psi})y^{-s}ds\\[5pt]
  =&-\frac{1}{4\pi}\frac{1}{\sqrt{y}}\tau_{\mathbb{Q}}(\overline{\sigma})
  \psi_{\mathrm{fin}}(p)T(\psi)\frac{\sigma\circ\mathbb{N}_{F/\mathbb{Q}}(\mathfrak{f})\tau_{F}(\sigma\circ\mathbb{N}_{F/\mathbb{Q}})}{p}\\
  &\hspace{50mm}\times\int^{i\infty}_{-i\infty}(D\mathbb{N}_{F/\mathbb{Q}}(p\mathfrak{f}))^{-s/2}
  \Lambda(1-s,(\overline{\sigma}\circ\mathbb{N}_{F/\mathbb{Q}})\overline{\psi})y^{-s}ds\\[5pt]
  =&\frac{1}{4\pi\sqrt{y}}\tau_{\mathbb{Q}}(\sigma)\psi_{\mathrm{fin}}(p)T(\psi)
  \frac{\sigma\circ\mathbb{N}_{F/\mathbb{Q}}(\mathfrak{f})\sigma(D)\tau_{\mathbb{Q}}(\sigma)^2}{p}\\
  &\hspace{50mm}\times\int^{i\infty}_{-i\infty}(D\mathbb{N}_{F/\mathbb{Q}}(p\mathfrak{f}))^{-s/2}
  \Lambda(1-s,(\overline{\sigma}\circ\mathbb{N}_{F/\mathbb{Q}})\overline{\psi})y^{-s}ds\\[5pt]
  =&\frac{1}{4\pi \sqrt{y}}\psi_{\mathrm{fin}}(p)T(\psi)\sigma(-D\mathbb{N}_{F/\mathbb{Q}}(\mathfrak{f}))
  \tau_{\mathbb{Q}}(\sigma)
  \int^{i\infty}_{-i\infty}(D\mathbb{N}_{F/\mathbb{Q}}(p\mathfrak{f}))^{-\frac{1-s}{2}}
  \Lambda(s,(\overline{\sigma}\circ\mathbb{N}_{F/\mathbb{Q}})\overline{\psi})y^{-\frac{1-s}{2}}ds\\[5pt]
  =&\frac{\psi_{\mathrm{fin}}(p)T(\psi)}{4\pi D\mathbb{N}_{F/\mathbb{Q}}(p\mathfrak{f})y^2}
  \sigma(-D\mathbb{N}_{F/\mathbb{Q}}(\mathfrak{f}))\tau_{\mathbb{Q}}(\sigma)
  \sqrt{yD\mathbb{N}_{F/\mathbb{Q}}(p\mathfrak{f})}\\
  &\hspace{40mm}\times\int_{-i\infty}^{i\infty}(D\mathbb{N}_{F/\mathbb{Q}}(p\mathfrak{f}))^{-s/2}
  \Lambda(s,(\overline{\sigma}\circ\mathbb{N}_{F/\mathbb{Q}})\overline{\psi})(D\mathbb{N}_{F/\mathbb{Q}}(p\mathfrak{f})y)^s ds\\[4pt]
  =&\frac{1}{y^2}\psi_{\mathrm{fin}}(p)T(\psi)\sigma(-D\mathbb{N}_{F/\mathbb{Q}}(\mathfrak{f}))
  \frac{\partial f_{\overline{\sigma},\overline{\psi}}}{\partial x}\left|
  \begin{pmatrix}
                              &-1\\
    D\mathbb{N}_{F/\mathbb{Q}}(p\mathfrak{f})&
  \end{pmatrix}\right.(iy)\\[4pt]
  =&-\psi_{\mathrm{fin}}(p)T(\psi)\sigma(-D\mathbb{N}_{F/\mathbb{Q}}(\mathfrak{f}))
  \left(\frac{\partial}{\partial x}f_{\overline{\sigma},\overline{\psi}}\left|
  \begin{pmatrix}
                              &-1\\
    D\mathbb{N}_{F/\mathbb{Q}}(p\mathfrak{f})&
  \end{pmatrix}\right.\right)(iy).
\end{align*}
From Lemma \ref{bump},
\[f_{\sigma,\psi}=-\psi_{\mathrm{fin}}(p)\sigma(-D\mathbb{N}_{F/\mathbb{Q}}(\mathfrak{f}))
T(\psi)f_{\overline{\sigma},\overline{\psi}}\left|
\begin{pmatrix}
                            &-1\\
  D\mathbb{N}_{F/\mathbb{Q}}(p\mathfrak{f})&                         
\end{pmatrix}\right. .
\]
Case 3. The case where $\sigma=1$.\\
A direct computation shows that
\begin{align*}
  (f_{\sigma,\psi}+\Theta_{\psi})(z)
  =&\sum_{b \bmod p}\Theta_{\psi}\left|
  \begin{pmatrix}
    p&b\\
     &p
  \end{pmatrix}\right.(z)\\
  =&\sum_{n\ne 0}\sum_{b\bmod p}a(n)\sqrt{y}K_{\nu}(2\pi|n|y)
  e^{2\pi i n(x+\frac{b}{p})}\\
  =&p\sum_{n\ne 0,p\mid n}a(n)\sqrt{y}
  K_{\nu}(2\pi |n|y)e^{2\pi i n x}\\
  =&p\sum_{\mathfrak{a}}\psi(p\mathfrak{a})\sqrt{y}
  K_{\nu}\left(2\pi\mathbb{N}_{F/\mathbb{Q}}(p\mathfrak{a})\right)
  \cos \left(2\pi \mathbb{N}_{F/\mathbb{Q}}(p\mathfrak{a})x\right)\\
  =&\psi_{\mathrm{fin}}(p)\sum_{\mathfrak{a}}\psi(\mathfrak{a})\sqrt{p^2 y}
  K_{\nu}\left(2\pi \mathbb{\mathfrak{a}}p^2 y\right)
  \cos\left(2\pi \mathbb{\mathfrak{a}}p^2 x\right)\\
  =&\psi_{\mathrm{fin}}(p)\Theta_{\psi}(p^2 z)\\
\end{align*}
Therefore, 
\[f_{\sigma,\psi}=-\Theta_{\psi}+\psi_{\mathrm{fin}}(p) \Theta_{\psi}\left|
\begin{pmatrix}
  p& \\
   &p^{-1}
\end{pmatrix}\right. .\]
Thus, we have 
\begin{align*}
  f_{\sigma,\psi}\left|
    \begin{pmatrix}
                                 &-1\\
       D\mathbb{N}_{F/\mathbb{Q}}(p\mathfrak{f})&
    \end{pmatrix}\right.
  =&-\Theta_{\psi}\left|
    \begin{pmatrix}
                                 &-1\\
       D\mathbb{N}_{F/\mathbb{Q}}(p\mathfrak{f})&
    \end{pmatrix}\right.
   +\psi_{\mathrm{fin}}(p)\Theta_{\psi}\left|
    \begin{pmatrix}
       p&      \\
        &p^{-1}
    \end{pmatrix}
    \begin{pmatrix}
                                 &-1\\
       D\mathbb{N}_{F/\mathbb{Q}}(p\mathfrak{f})&
    \end{pmatrix}\right.\\[5pt]
  =&-\Theta_{\psi}\left|
    \begin{pmatrix}
                                 &-1\\
       D\mathbb{N}_{F/\mathbb{Q}}(p\mathfrak{f})&
    \end{pmatrix}
    \begin{pmatrix}
      p&      \\
       &p^{-1}
    \end{pmatrix}\right.
   +\psi_{\mathrm{fin}}(p)\Theta_{\psi}\left|
     \begin{pmatrix}
                                 &-p\\
       Dp\mathbb{N}_{F/\mathbb{Q}}(\mathfrak{f})&
    \end{pmatrix}\right.\\[5pt]
  =&-T(\psi)\Theta_{\overline{\psi}}\left|
   \begin{pmatrix}
       p&      \\
        &p^{-1}
   \end{pmatrix}\right.
   +\psi_{\mathrm{fin}}(p)T(\psi)\Theta_{\overline{\psi}}\\
  =&-T(\psi)\psi_{\mathrm{fin}}(p)\left(\overline{\psi}_{\mathrm{fin}}(p)
  \Theta_{\overline{\psi}}\left|
    \begin{pmatrix}
      p&    \\
       &p^{-1}
    \end{pmatrix}\right.
  -\Theta_{\overline{\psi}}\right)\\[5pt]
  =&-T(\psi)\psi_{\mathrm{fin}}(p)f_{\overline{\sigma},\overline{\psi}}\\[3pt]
  =&-T(\psi)\psi_{\mathrm{fin}}(p)
  \sigma(-D\mathbb{N}_{F/\mathbb{Q}}(\mathfrak{f}))f_{\overline{\sigma},\overline{\psi}}.
  \end{align*}
From the Cases 1-3, we derive Equation \eqref{f;automorphy2}.
\end{proof}
Using Equation \eqref{f;automorphy1},
we establish the automorphy of $\Theta_{\psi}$.
Take $m,r\in \mathbb{Z}$ such that
\[
  -D \mathbb{N}_{F/\mathbb{Q}}(\mathfrak{f}) m r \equiv 1 \pmod{p}
\]
and $s\in\mathbb{Z}$ such that
\[ps-D\mathbb{N}_{F/\mathbb{Q}}(\mathfrak{f})mr=1.\]
Set $\sigma(n)$ by 
\[
  \sigma(n) =
  \begin{cases}
    1 & (\text{if }n \equiv m \pmod{p}), \\[4pt]
    0 & (\text{if }n \not\equiv m \pmod{p}),
  \end{cases}
\]
namely, $\sigma(n)$ takes $1$ if $n \equiv m \pmod{p}$ and $0$ otherwise.
then, $\rho$ is the following map:
\[\rho (n)=\begin{cases}
            1&(\text{if }n\equiv r \pmod p)\\[4pt]
            0&(\text{if }n\not\equiv r \pmod p).
          \end{cases}\]
Applying Equation \eqref{f;automorphy2},
\[
  \Theta_{\psi}\left|
    \begin{pmatrix}
      p&m\\
       &p
    \end{pmatrix}
    \begin{pmatrix}
                                &-1\\
      D\mathbb{N}_{F/\mathbb{Q}}(p\mathfrak{f})&
    \end{pmatrix}\right.
  = -T(\psi)\psi_{\mathrm{fin}}(p)\Theta_{\overline{\psi}}\left|
    \begin{pmatrix}
      p&r\\
       &p
    \end{pmatrix}\right. .
 \]
 Therefore,
\begin{align*}
  \Theta_{\psi}\left|\begin{pmatrix}
    p&m\\
     &p
  \end{pmatrix}\right.
  =&-T(\psi)\psi_{\mathrm{fin}}(p)
    \Theta_{\overline{\psi}}\left|
      \begin{pmatrix}
        p&r\\
         &p
      \end{pmatrix}
      \begin{pmatrix}
                                  &-1\\
        D\mathbb{N}_{F/\mathbb{Q}}(p\mathfrak{f})&
      \end{pmatrix}\right.\\
  =&-\psi_{\mathrm{fin}}(p)\Theta_{\psi}\left|
    \begin{pmatrix}
                                  &-1\\
       D\mathbb{N}_{F/\mathbb{Q}}(\mathfrak{f})  &
    \end{pmatrix}
    \begin{pmatrix}
      p&r\\
       &p
    \end{pmatrix}
    \begin{pmatrix}
                                 &-1\\
       D\mathbb{N}_{F/\mathbb{Q}}(p\mathfrak{f})
    \end{pmatrix}\right.\\
  =&-\psi_{\mathrm{fin}}(p)\Theta_{\psi}\left|
    \begin{pmatrix}
                                p&-m\\
      -Dr\mathbb{N}_{F/\mathbb{Q}}(\mathfrak{f})&s
    \end{pmatrix}\right.
    \begin{pmatrix}
      p&m\\ &p
    \end{pmatrix}
\end{align*}
Thus, we have
\begin{equation}
  \Theta_{\psi}
  = -\psi_{\mathrm{fin}}(p)\Theta_{{\psi}}\left|
    \begin{pmatrix}
                                p&-m\\
      -Dr\mathbb{N}_{F/\mathbb{Q}}(\mathfrak{f})& s
    \end{pmatrix}\right. .
  \label{automorphy1}
\end{equation}
For $\begin{pmatrix}
  a&b\\c&d
\end{pmatrix}\in \Gamma_{0}(D\mathbb{N}_{F/\mathbb{Q}}(\mathfrak{f}))$
with $\chi_{D}(a)=-1$, we can take $u\in\mathbb{Z}$
such that 
\[
  \begin{pmatrix}
    1&u\\ &1
  \end{pmatrix}
  \begin{pmatrix}
    a&b\\c&d
  \end{pmatrix}
  =\begin{pmatrix}
    a+uc&b+ud\\c&d
  \end{pmatrix}
  =\begin{pmatrix}
    p&\ast\\\ast&\ast
  \end{pmatrix}
  \qquad(\text{for some prime number }p).
\]
Since $\chi_{D}(a)=-1$, it follows that $\chi_{D}(p)=-1$.
Thus, from Equation \eqref{automorphy1},  we have
\begin{align*}
 \Theta_{\psi}
   =& -\psi_{\mathrm{fin}}(p)\Theta_{\psi}\left|\begin{pmatrix}
    a+uc&b+ud\\c&d
   \end{pmatrix}\right.\\
  =&-\psi_{\mathrm{fin}}(p)\Theta_{\psi}\left|\begin{pmatrix}
    a&b\\c&d
  \end{pmatrix}\right.\\
  =&\chi_{D}(a)\psi_{\mathrm{fin}}(a)\Theta_{\psi}\left|\begin{pmatrix}
    a&b\\c&d
  \end{pmatrix}\right. .
\end{align*}

Therefore, we obtain
\begin{equation} \Theta_{\psi}\left|\begin{pmatrix}
  a&b\\c&d
\end{pmatrix}\right.
=\chi_{D}(d)\psi_{\mathrm{fin}}(d)\Theta_{\psi}\label{automorphy2}
\end{equation}
for
$\begin{pmatrix}
  a&b\\c&d
\end{pmatrix}\in\Gamma_{0}(D\mathbb{N}_{F/\mathbb{Q}}(\mathfrak{f}))$
with $\chi_{D}(a)=-1$.
Any element of $\Gamma_{0}(D\mathbb{N}_{F/\mathbb{Q}}(\mathfrak{f}))$
can be written as a product of elements
$\begin{pmatrix}
  a&b\\c&d
\end{pmatrix}
\in\Gamma_{0}(D\mathbb{N}_{F/\mathbb{Q}}(\mathfrak{f}))$
with $\chi_{D}(a)=-1$, thus Equation \eqref{automorphy2}
holds for any element of
$\Gamma_{0}(D\mathbb{N}_{F/\mathbb{Q}}(\mathfrak{f}))$.

Next we prove the cuspidality.
Let $l\in \mathbb{Z},|l|\geq 2$, and
let $\sigma :\left(\mathbb{Z}/l\mathbb{Z}\right)^\times\to S^1$
be a Dirichlet character. If $l$ is not prime number, we also define
$f_{\sigma,\psi}$ by Equation \eqref{f}.
Suppose that $\psi$ is not of the form $\sigma\circ\mathbb{N}_{F/\mathbb{Q}}$.
If $\sigma(-1)=-1$, it follows that $f_{\sigma,\psi}=0$. If $\sigma (-1)=1$,
we denote by $\psi'$
the primitive Hecke character induced by $(\sigma\circ\mathbb{N}_{F/\mathbb{Q}})\psi$.
and by $\mathfrak{f}'$ the conductor of $\psi'$.
Note that $\psi '\ne 1$. One can verify that
\begin{align*}
  f_{\sigma,\psi}(z)
  =&\sum_{\substack{m\bmod l\\(m,l)=1}}\overline{\sigma}(m)\sum_{n\ne 0}
  a(n)\sqrt{y}K_{\nu}(2\pi |n|y)e^{2\pi i n x}e^{2\pi i\frac{mn}{l}}\\
  =&\sum_{n\ne 0}\sigma(n)\tau_{\mathbb{Q}}(\overline{\sigma})a(n)\sqrt{y}
  K_{\nu}(2\pi|n|y)e^{2\pi inx}\\
  =&\tau_{\mathbb{Q}}(\overline{\sigma})\sum_{\mathfrak{a}}
  \sigma\circ\mathbb{N}_{F/\mathbb{Q}}(\mathfrak{a})\psi(\mathfrak{a})\sqrt{y}K_{\nu}(2\pi\mathbb{N}_{F/\mathbb{Q}}(\mathfrak{a})y)
  \cos (2\pi\mathbb{N}_{F/\mathbb{Q}}(\mathfrak{a})x).
\end{align*}
Thus the Mellin transform of $f_{\sigma,\psi}$ becomes 
\begin{align*}
  \int_{0}^{\infty}f_{\sigma,\psi}(iy)y^{s-1/2}\frac{dy}{y}
  =&\tau_{\mathbb{Q}}(\overline{\sigma})\sum_{\mathfrak{a}}\sigma\circ\mathbb{N}_{F/\mathbb{Q}}(\mathfrak{a})\psi(\mathfrak{a})
  \int_{0}^{\infty}K_{\nu}(2\pi\mathbb{N}_{F/\mathbb{Q}}(\mathfrak{a})y)y^s\frac{dy}{y}\\
  =&\tau_{\mathbb{Q}}(\overline{\sigma})\sum_{\mathfrak{a}}
  \sigma\circ\mathbb{N}_{F/\mathbb{Q}}(\mathfrak{a})\psi(\mathfrak{a})\Gamma\left(\frac{s-\nu}{2}\right)
  \Gamma\left(\frac{s+\nu}{2}\right)2^{s-2}(2\pi\mathbb{N}_{F/\mathbb{Q}}(\mathfrak{a}))^{-s}\\
  =&\frac{1}{4}\tau_{\mathbb{Q}}(\overline{\sigma})L(s,(\sigma\circ\mathbb{N}_{F/\mathbb{Q}})\psi)
  \Gamma\left(\frac{s-\nu}{2}\right)\Gamma\left(\frac{s+\nu}{2}\right)\pi^{-s}\\
  =&\frac{1}{4}\tau_{\mathbb{Q}}(\overline{\sigma})
  \prod_{\mathfrak{p}\mid l\mathfrak{f}}\left(1-\psi '
  (\mathfrak{p})\mathbb{N}_{F/\mathbb{Q}}(\mathfrak{p})\right)
  L(s,(\sigma\circ\mathbb{N}_{F/\mathbb{Q}})\psi)
  \Gamma\left(\frac{s-\nu}{2}\right)\Gamma\left(\frac{s+\nu}{2}\right)\pi^{-s}\\
  =&\frac{1}{4}\tau_{\mathbb{Q}}(\overline{\sigma})(D\mathbb{N}_{F/\mathbb{Q}}(\mathfrak{f}'))^{-s/2}
  \Lambda(s,\psi ')E(s,(\sigma\circ\mathbb{N}_{F/\mathbb{Q}})\psi).
\end{align*}
Here, we put 
$\displaystyle E(s,(\sigma\circ\mathbb{N}_{F/\mathbb{Q}})\psi)=\prod_{\mathfrak{p}\mid l\mathfrak{f}}\left(1-\psi '
  (\mathfrak{p})\mathbb{N}_{F/\mathbb{Q}}(\mathfrak{p})\right)$.\\
By the Mellin inversion formula and Remark \ref{remark;Pharagmen-Lindelöf},
we obtain the following asymptotic estimate as $y\rightarrow +0$:
\begin{align*}
f_{\sigma,\psi}(iy)
=&\frac{\sqrt{y}}{8\pi i}\tau_{\mathbb{Q}}(\overline{\sigma})
\int_{R+i\infty}^{R-i\infty}(D\mathbb{N(\mathfrak{f}')})^{-s/2}
\Lambda(s,\psi ')E(s,(\sigma\circ\mathbb{N}_{F/\mathbb{Q}})\psi)y^{-s}ds\\
=& O_{t}(y^t)
\end{align*}
for any $t\in\mathbb{R}$ and $R>1$.
Therefore $f_{\sigma,\psi}(z)$
decays rapidly as $y\rightarrow +0$ for any
$\sigma\colon\left(\mathbb{Z}/l\mathbb{Z}\right)^\times\to S^1$.
Hence we have 
\[\Theta_{\psi}\left|\begin{pmatrix}
  l&m\\
   &l
\end{pmatrix}\right.
=O_t (y^t),\qquad y\rightarrow +0
\]
for any $t\in\mathbb{R}$ and integers
$l,m$ with $(l,m)=1$.
From the argument above, 
if $\psi$ is not of the form $\sigma\circ\mathbb{N}_{F/\mathbb{Q}}$,
then $\Theta_{\psi}$ is a Maass wave cusp form for 
$\Gamma_{0}(D\mathbb{N}_{F/\mathbb{Q}}(\mathfrak{f}))$.
If $\psi=\sigma\circ\mathbb{N}_{F/\mathbb{Q}}$ for some Dirichlet
character $\sigma$, then $\Lambda(s,\psi ')$ has a pole at $s=1$, hence
$t$ has upper bound $-\frac{1}{2}$. Therefore $\Theta_{\psi}$ is 
Maass wave form but not a cusp form.
\subsection{Proof for the case $\epsilon=1$}
Next, we show in the case  $\epsilon =1$. As mentioned above, 
the method is basically same as in the case of $\epsilon =0$. 
We perform a precise computation 
noting differences in the parity.
At first we show the following lemma. 

\begin{lemma}
  Let $p$ be an inert prime with $p\nmid \mathbb{N}_{F/\mathbb{Q}}(\mathfrak{f})$.
  Define a function $f_{\sigma,\psi}$ on $\mathbb{H}$ by Equation \eqref{f}
  for any map $\sigma :\left(\mathbb{Z}/p\mathbb{Z}\right)^\times\to S^1$.
  Then it holds that 
\begin{enumerate}
  \item \begin{equation}\label{eq;automorphy of theta1}
        \Theta_{\psi}=-T(\psi)\Theta_{\overline{\psi}}\left|\begin{pmatrix}
                           &-1\\
         D\mathbb{N}_{F/\mathbb{Q}}(\mathfrak{f})&
        \end{pmatrix}\right. ,
        \end{equation}
  \item \begin{equation}\label{f;automorphy3}
         f_{\sigma,\psi}\left|\begin{pmatrix}
                            &-1\\
         D\mathbb{N}_{F/\mathbb{Q}}(p\mathfrak{f})&
         \end{pmatrix}\right.
         =T(\psi)\psi_{\mathrm{fin}}(p)f_{\rho,\overline{\psi}}
        \end{equation}
        where a map 
        $\rho\colon \left(\mathbb{Z}/p\mathbb{Z}\right)^\times\to S^1$
        is defined by 
        $\rho(m) \coloneq \sigma(r)$
          with $-rmD \mathbb{N}_{F/\mathbb{Q}}(\mathfrak{f}) 
           \equiv 1 \bmod p$.
\end{enumerate}
\end{lemma}
\begin{proof}
We compute the Mellin transform of 
$\displaystyle\frac{\partial \theta_{\psi}}{\partial x}(iy)$:
\begin{align*}
  \int_{0}^{\infty}\frac{\partial\Theta_{\psi}}{\partial x}(iy)
  y^{s+1/2}\frac{dy}{y}
  =&\sum_{\mathfrak{a}}\psi(\mathfrak{a})(2\pi \mathbb{N}_{F/\mathbb{Q}}(\mathfrak{a}))
  \int_{0}^{\infty}K_{\nu}(2\pi \mathbb{N}_{F/\mathbb{Q}}(\mathfrak{a})y)y^{s+1}\frac{dy}{y}\\
  =&\sum_{_{\mathfrak{a}}}\psi(\mathfrak{a})(2\pi\mathbb{N}_{F/\mathbb{Q}}(\mathfrak{a}))^{-s}
  \int_{0}^{\infty}K_{\nu}(y)y^{s+1}\frac{dy}{y}\\
  =&\sum_{\mathfrak{a}}\psi(\mathfrak{a})(2\pi \mathbb{N}_{F/\mathbb{Q}}(\mathfrak{a}))^{-s}
  2^{s-1}\Gamma\left(\frac{s+1-\nu}{2}\right)\Gamma\left(\frac{s+1+\nu}{2}\right)\\
  =&\frac{1}{2}\pi ^{-s}L(s,\psi)\Gamma\left(\frac{s+1-\nu}{2}\right)
  \Gamma\left(\frac{s+1+\nu}{2}\right)\\
  =&\frac{1}{2}(D\mathbb{N}_{F/\mathbb{Q}}(\mathfrak{f}))^{-s/2}\Lambda(s,\psi).
\end{align*}
Thus by the same argument in the case $\epsilon=0$, we have
\begin{align*}
  \frac{\partial \Theta_{\psi}}{\partial x}(iy)
  =&\frac{1}{4\pi i}\frac{1}{\sqrt{y}}\int_{-i\infty}^{i\infty}
  (D\mathbb{N}_{F/\mathbb{Q}}(\mathfrak{f}))^{-s/2}\Lambda(s,\psi)y^{-s}ds\\
  =&\frac{1}{4\pi i}\frac{1}{\sqrt{y}}T(\psi)
  \int_{-i\infty}^{i\infty}
  (D\mathbb{N}_{F/\mathbb{Q}}(\mathfrak{f}))^{-s/2}\Lambda(1-s,\overline{\psi})y^{-s}ds\\
  =&\frac{1}{4\pi i}\frac{1}{\sqrt{y}}T(\psi)
  \int_{-i\infty}^{i\infty}
  (D\mathbb{N}_{F/\mathbb{Q}}(\mathfrak{f}))^{-\frac{1-s}{2}}
  \Lambda(s,\overline{\psi})y^{-1+s}ds\\
  =&\frac{1}{4\pi i}\frac{T(\psi)}{y^2 D\mathbb{N}_{F/\mathbb{Q}}(\mathfrak{f})}
  \sqrt{yD\mathbb{N}_{F/\mathbb{Q}}(\mathfrak{f})}\int_{-i\infty}^{i\infty}
  (D\mathbb{N}_{F/\mathbb{Q}}(\mathfrak{f}))^{-s/2}\Lambda(s,\overline{\psi})(D\mathbb{N}_{F/\mathbb{Q}}(\mathfrak{f})y)^s ds\\
  =&T(\psi)\frac{1}{y^2 D\mathbb{N}_{F/\mathbb{Q}}(\mathfrak{f})}
  \frac{\partial \Theta_{\overline{\psi}}}{\partial x}
  \left(\frac{i}{D\mathbb{N}_{F/\mathbb{Q}}(\mathfrak{f})y}\right)\\
  =&-T(\psi)\left(\frac{\partial}{\partial x}\left|\begin{pmatrix}
                             &-1\\
    D\mathbb{N}_{F/\mathbb{Q}}(\mathfrak{f})&
  \end{pmatrix}\right.\right)(iy).
\end{align*}
By Lemma \ref{bump}, we obtain 
Equation \eqref{eq;automorphy of theta1}.\\
Next we show Equation \eqref{f;automorphy3}. 
Since the linear space consisting of all maps 
$\left(\mathbb{Z}/p\mathbb{Z}\right)^\times\to S^1$
is spanned by Dirichlet characters,
we may assume that the map $\sigma$ is a Dirichlet character.
In this case Equation \eqref{f;automorphy3} is equivalent to
the following Equation \eqref{f;automorphy4}:
\begin{equation}\label{f;automorphy4}
  f_{\sigma,\psi}\left|\begin{pmatrix}
                              &-1\\
    D\mathbb{N}_{F/\mathbb{Q}}(p\mathfrak{f})&
  \end{pmatrix}\right.
  =T(\psi)\psi_{\mathrm{fin}}(p)\sigma(-D\mathbb{N}_{F/\mathbb{Q}}(\mathfrak{f}))
  f_{\overline{\sigma},\overline{\psi}}
\end{equation}
We divide the following three cases.\\
Case 1. The case where $\sigma$ is primitive and $\sigma(-1)=1$\\
Let $a(n)$ be the Fourier coefficient
of $\Theta_\psi$ defined by Equation \eqref{fourier coefficient}. Then,
\begin{align*}
  f_{\sigma,\psi}(z)
  =&\sum_{\substack{m\bmod p\\p\nmid m}}\overline{\sigma}(m)\sum_{n\ne 0}a(n)\sqrt{y}
  K_{\nu}(2\pi |n|y)e^{2\pi i nx}e^{2\pi in \frac{m}{p}}\\
  =&\sum_{n\ne 0}a(n)\sigma(n)\tau_{\mathbb{Q}}(\overline{\sigma})\sqrt{y}
  K_{\nu}(2\pi |n|y)e^{2\pi inx}\\
  =&\tau_{\mathbb{Q}}(\overline{\sigma})\sum_{\mathfrak{a}}
  \psi(\mathfrak{a})\sigma\circ\mathbb{N}_{F/\mathbb{Q}}(\mathfrak{a})\sqrt{y}
  K_\nu (2\pi \mathbb{N}_{F/\mathbb{Q}}(\mathfrak{a})y)\sin (2\pi\mathbb{N}_{F/\mathbb{Q}}(\mathfrak{a})x) 
\end{align*}
Taking partial derivative with respect to $x$,  we have
\[\frac{\partial f_{\sigma,\psi}}{\partial x}(z)
=\tau(\overline{\sigma})\sum_{\mathfrak{a}}\psi(\mathfrak{a})
\sigma\circ\mathbb{N}_{F/\mathbb{Q}}(\mathfrak{a})\sqrt{y}
K_\nu(2\pi \mathbb{N}_{F/\mathbb{Q}}(\mathfrak{a})y)(2\pi \mathbb{N}_{F/\mathbb{Q}}(\mathfrak{a}))
\cos (2\pi\mathbb{N}_{F/\mathbb{Q}}(\mathfrak{a})x).\]
Thus we compute its Mellin transform as follows:
\begin{align*}
  \int_{0}^{\infty}\frac{\partial f_{\sigma,\psi}}{\partial x}(iy)
  y^{s+1/2}\frac{dy}{y}
  =&\tau_{\mathbb{Q}}(\overline{\sigma})\sum_{\mathfrak{a}}
  \psi(\mathfrak{a})\sigma\circ\mathbb{N}_{F/\mathbb{Q}}(\mathfrak{a})(2\pi\mathbb{N}_{F/\mathbb{Q}}(\mathfrak{a}))
  \int_{0}^{\infty}K_{\nu}(2\pi\mathbb{N}_{F/\mathbb{Q}}(\mathfrak{a})y)y^{s+1}\frac{dy}{y}\\
  =&\tau_{\mathbb{Q}}(\overline{\sigma})\sum_{\mathfrak{a}}
  \psi(\mathfrak{a})\sigma\circ\mathbb{N}_{F/\mathbb{Q}}(\mathfrak{a})(2\pi \mathbb{N}_{F/\mathbb{Q}}(\mathfrak{a}))^{-s}
  2^{s-1}\Gamma\left(\frac{s+1-\nu}{2}\right)\Gamma\left(\frac{s+1+\nu}{2}\right)\\
  =&\frac{1}{2}\tau_{\mathbb{Q}}(\overline{\sigma})\pi ^{-s}L(s,(\sigma\circ\mathbb{N}_{F/\mathbb{Q}})\psi)
  \Gamma\left(\frac{s+1-\nu}{2}\right)\Gamma\left(\frac{s+1+\nu}{2}\right)\\
  =&\frac{1}{2}\tau_{\mathbb{Q}}(\overline{\sigma})(D\mathbb{N}_{F/\mathbb{Q}}(p\mathfrak{f}))^{-s/2}
  \Lambda(s,(\sigma\circ\mathbb{N}_{F/\mathbb{Q}})\psi).
\end{align*}
From the Mellin inversion formula,
Theorem \ref{theorem;Hecke L function}, and
Remark \ref{remark;Pharagmen-Lindelöf}, we have
\begin{align}
  \frac{\partial f_{\sigma,\psi}}{\partial x}(iy)
  =&\frac{1}{4\pi i}\frac{1}{\sqrt{y}}\tau_{\mathbb{Q}}(\overline{\sigma})
  \int_{-i\infty}^{i\infty}(D\mathbb{N}_{F/\mathbb{Q}}(p\mathfrak{f}))^{-s/2}
  \Lambda(s,(\sigma\circ\mathbb{N}_{F/\mathbb{Q}})\psi)y^{-s}ds\notag\notag\\
  =&\frac{1}{4\pi i}\frac{1}{\sqrt{y}}\tau_{\mathbb{Q}}(\overline{\sigma})
  T((\sigma\circ\mathbb{N}_{F/\mathbb{Q}})\psi)\int_{-i\infty}^{i\infty}
  (D\mathbb{N}_{F/\mathbb{Q}}(p\mathfrak{f}))^{-s/2}\Lambda(1-s,\overline{\sigma}\circ\mathbb{N}_{F/\mathbb{Q}}\overline{\psi})
  y^{-s}ds\notag\\
  =&\frac{1}{4\pi i}\frac{1}{\sqrt{y}}\tau_{\mathbb{Q}}(\overline{\sigma})
  T((\sigma\circ\mathbb{N}_{F/\mathbb{Q}})\psi)\int_{-i\infty}^{i\infty}
  (D\mathbb{N}_{F/\mathbb{Q}}(p\mathfrak{f}))^{-\frac{1-s}{2}}
  \Lambda(s,(\overline{\sigma}\circ\mathbb{N}_{F/\mathbb{Q}})\overline{\psi})y^{-1+s}ds\notag\\
  =&\frac{1}{4\pi i}\frac{1}{D\mathbb{N}_{F/\mathbb{Q}}(p\mathfrak{f})y^2}
  \tau_{\mathbb{Q}}(\overline{\sigma})T((\sigma\circ\mathbb{N}_{F/\mathbb{Q}})\psi)
  \sqrt{yD\mathbb{N}_{F/\mathbb{Q}}(p\mathfrak{f})}\\
  &\hspace{30mm}\times\int_{-i\infty}^{i\infty}
  (D\mathbb{N}_{F/\mathbb{Q}}(p\mathfrak{f}))^{-s/2}
  \Lambda(s,(\overline{\sigma}\circ\mathbb{N}_{F/\mathbb{Q}})\overline{\psi})
  (D\mathbb{N}_{F/\mathbb{Q}}(p\mathfrak{f})y)^s ds. \label{Mellin inversion2}
\end{align}
Using Lemma \ref{absolute value of Gauss1}, 
the following relation for Gauss sums holds:
\begin{align}
  T((\sigma\circ\mathbb{N}_{F/\mathbb{Q}})\psi)
  =&-\frac{\tau_{F}((\sigma\circ\mathbb{N}_{F/\mathbb{Q}})\psi)}{\mathbb{N}_{F/\mathbb{Q}}(p\mathfrak{f})^{1/2}}
  =-\frac{\sigma\circ\mathbb{N}_{F/\mathbb{Q}}(\mathfrak{f})\psi_{\mathrm{fin}}(p)\tau_{F}(\sigma\circ\mathbb{N}_{F/\mathbb{Q}})\tau_{F}(\psi)}{p\mathbb{N}_{F/\mathbb{Q}}(\mathfrak{f})^{1/2}}\notag\\
  =&\frac{1}{p}\sigma\circ\mathbb{N}_{F/\mathbb{Q}}(\mathfrak{f})
  \psi_{\mathrm{fin}}(p)\tau_{F}(\sigma\circ\mathbb{N}_{F/\mathbb{Q}})T(\psi)\notag\\
  =&\frac{1}{p}\sigma\circ\mathbb{N}_{F/\mathbb{Q}}(\mathfrak{f})\psi_{\mathrm{fin}}(p)
  (\sigma(D)\chi_{D}(p)\tau_{\mathbb{Q}}(\sigma)^2)T(\psi)\notag\\
  =&\frac{1}{p}\sigma(D\mathbb{N}_{F/\mathbb{Q}}(\mathfrak{f}))\psi_{\mathrm{fin}}(p)
  \chi_{D}(p)\tau_{\mathbb{Q}}(\sigma)^2T(\psi).\label{gauss}
\end{align}
Applying Equation \eqref{gauss} to Equation \eqref{Mellin inversion2}, 
we derive
\begin{align*}
  \frac{\partial f_{\sigma,\psi}}{\partial x}(iy)
  =&\frac{1}{4\pi i}\frac{1}{D\mathbb{N}_{F/\mathbb{Q}}(p\mathfrak{f})y^2}
  \tau_{\mathbb{Q}}(\overline{\sigma})
  \left(\frac{1}{p}\sigma(D\mathbb{N}_{F/\mathbb{Q}}(\mathfrak{f}))
  \psi_{\mathrm{fin}}(p)\chi_{D}(p)\tau_{\mathbb{Q}}(\sigma)^2T(\psi)\right)\\
  &\hspace{25mm}\times\sqrt{yD\mathbb{N}_{F/\mathbb{Q}}(p\mathfrak{f})}
  \int_{-i\infty}^{i\infty}
  (D\mathbb(p\mathfrak{f}))^{-s/2}
  \Lambda(s,(\overline{\sigma\circ\mathbb{N}_{F/\mathbb{Q}}})\overline{\psi})
  (D\mathbb{N}_{F/\mathbb{Q}}(p\mathfrak{f})y)^s ds \\
  =&-\frac{1}{4\pi i}\frac{1}{D\mathbb{N}_{F/\mathbb{Q}}(p\mathfrak{f})y^2}
  \sigma(D\mathbb{N}_{F/\mathbb{Q}}(\mathfrak{f}))\psi_{\mathrm{fin}}(p)\tau_{\mathbb{Q}}(\sigma)T(\psi)
  \sqrt{yD\mathbb{N}_{F/\mathbb{Q}}(p\mathfrak{f})}\\
  &\hspace{35mm}\times\int_{-i\infty}^{i\infty}
  (D\mathbb{N}_{F/\mathbb{Q}}(p\mathfrak{f}))^{-s/2}
  \Lambda(s,(\overline{\sigma}\circ\mathbb{N}_{F/\mathbb{Q}})\overline{\psi})
  (D\mathbb{N}_{F/\mathbb{Q}}(p\mathfrak{f})y)^s ds\\
  =&-\frac{1}{D\mathbb{N}_{F/\mathbb{Q}}(p\mathfrak{f})y^2}\sigma(-D\mathbb{N}_{F/\mathbb{Q}}(\mathfrak{f}))
  \psi_{\mathrm{fin}}(p)T(\psi)\frac{\partial f_{\overline{\sigma},\overline{\psi}}}{\partial x}
  \left(\frac{i}{D\mathbb{N}_{F/\mathbb{Q}}(p\mathfrak{f})y}\right)\\
  =&\sigma(-D\mathbb{N}_{F/\mathbb{Q}}(\mathfrak{f}))\psi_{\mathrm{fin}}(p)T(\psi)
  \frac{\partial}{\partial x}
  \left(f_{\overline{\sigma},\overline{\psi}}\left|\begin{pmatrix}
                                  &-1\\
        D\mathbb{N}_{F/\mathbb{Q}}(p\mathfrak{f})&
  \end{pmatrix}\right.\right)(iy).
\end{align*}
By Lemma \ref{bump}, we obtain Equation \eqref{f;automorphy4}.\\
Case 2. In the case of $\sigma$ is primitive and $\sigma(-1)=-1$.\\
By the following computation using the Fourier coefficients 
$a(n)$ of $f_{\sigma, \psi}$, we have
\begin{align*}
  f_{\sigma,\psi}(z)
  =&\sum_{\substack{m\bmod p\\p\nmid m}}\overline{\sigma}(m)\sum_{n\ne 0}a(n)\sqrt{y}
  K_{\nu}(2\pi |n| y)e^{2\pi nx}e^{2\pi in\frac{m}{p}}\\
  =&\tau_{\mathbb{Q}}(\overline{\sigma})\sum_{n\ne 0}a(n)\sqrt{y}K_{\nu}(2\pi |n|y)e^{2\pi inx}\\
  =&-i\tau_{\mathbb{Q}}(\overline{\sigma})\sum_{\mathfrak{a}}
  \sigma\circ\mathbb{N}_{F/\mathbb{Q}}(\mathfrak{a})\psi(\mathfrak{a})\sqrt{y}
  K_{\nu}(2\pi\mathbb{N}_{F/\mathbb{Q}}(\mathfrak{a})y)\cos (2\pi \mathbb{N}_{F/\mathbb{Q}}(\mathfrak{a})x).
\end{align*}
Then its Mellin transform is given by
\begin{align*}
  \int_{0}^{\infty}f_{\sigma,\psi}(iy)y^{s-1/2}\frac{dy}{y}
  =&-i\tau_{\mathbb{Q}}(\overline{\sigma})\sum_{\mathfrak{a}}\sigma\circ\mathbb{N}_{F/\mathbb{Q}}(\mathfrak{a})
  \psi(\mathfrak{a})\int_{o}^{\infty}K_{\nu}(2\pi\mathbb{N}_{F/\mathbb{Q}}(\mathfrak{a})y)
  y^s ds\\
  =&-\tau_{\mathbb{Q}}(\overline{\sigma})\sum_{\mathfrak{a}}\psi(\mathfrak{a})
  (2\pi\mathbb{N}_{F/\mathbb{Q}}(\mathfrak{a}))^{-s}2^{s-2}
  \Gamma\left(\frac{s-\nu}{2}\right)\Gamma\left(\frac{s+\nu}{2}\right)\\
  =&-\frac{\tau_{\mathbb{Q}}(\overline{\sigma})}{4i}
  \pi ^{-s}L(s,(\sigma\circ\mathbb{N}_{F/\mathbb{Q}})\psi)
  \Gamma\left(\frac{s-\nu}{2}\right)\Gamma\left(\frac{s+\nu}{2}\right)\\
  =&\frac{\tau_{\mathbb{Q}}(\overline{\sigma})}{4i}(D\mathbb{N}_{F/\mathbb{Q}}(p\mathfrak{f}))^{-s/2}
  \Lambda(s,(\sigma\circ\mathbb{N}_{F/\mathbb{Q}})\psi).
\end{align*}
By the Mellin inversion formula, Remark \ref{remark;Pharagmen-Lindelöf}, and 
Theorem \ref{theorem;Hecke L function}, we have 
\begin{align*}
  f_{\sigma,\psi}(iy)
  =&-\frac{\tau_{\mathbb{Q}}(\overline{\sigma})}{8\pi}\sqrt{y}
  \int_{-i\infty}^{i\infty}(D\mathbb{N}_{F/\mathbb{Q}}(p\mathfrak{f}))^{-s/2}
  \Lambda(s,(\sigma\circ\mathbb{N}_{F/\mathbb{Q}})\psi)y^{-s}ds\\
  =&-\frac{\tau_{\mathbb{Q}}(\overline{\sigma})}{8\pi}\sqrt{y}
  T((\sigma\circ\mathbb{N}_{F/\mathbb{Q}})\psi)
  \int_{-i\infty}^{i\infty}(D\mathbb{N}_{F/\mathbb{Q}}(p\mathfrak{f}))^{-s/2}
  \Lambda(1-s,(\overline{\sigma}\circ\mathbb{N}_{F/\mathbb{Q}})\overline{\psi})y^{-s}ds\\
  =&-\frac{\tau_{\mathbb{Q}}(\overline{\sigma})}{8\pi}\sqrt{y}T((\sigma\circ\mathbb{N}_{F/\mathbb{Q}})\psi)
  \int_{-i\infty}^{i\infty}(D\mathbb{N}_{F/\mathbb{Q}}(p\mathfrak{f}))^{-\frac{1-s}{2}}
  \Lambda(s,(\overline{\sigma}\circ\mathbb{N}_{F/\mathbb{Q}})\overline{\psi})y^{-1+s}ds\\
  =&-\frac{\tau_{\mathbb{Q}}(\overline{\sigma})}{8\pi}
  \frac{1}{\sqrt{yD\mathbb{N}_{F/\mathbb{Q}}(\mathfrak{f})}}T((\sigma\circ\mathbb{N}_{F/\mathbb{Q}})\psi)\\
  &\hspace{30mm}\times\int_{-i\infty}^{i\infty}(D\mathbb{N}_{F/\mathbb{Q}}(p\mathfrak{f}))^{-s/2}
  \Lambda(s,(\overline{\sigma}\circ\mathbb{N}_{F/\mathbb{Q}})\overline{\psi})
  (D\mathbb{N}_{F/\mathbb{Q}}(p\mathfrak{f})y)^s ds.
\end{align*}
From Equation (\ref{gauss}), we obtain
\begin{align*}
  f_{\sigma,\psi}(iy)
  =&-\frac{\tau_{\mathbb{Q}}(\mathbb{\sigma})}{8\pi}
  \frac{1}{\sqrt{yD\mathbb{N}_{F/\mathbb{Q}}(p\mathfrak{f})}}
  \left(-\sigma(D\mathbb{N}_{F/\mathbb{Q}}(\mathfrak{f}))\psi_{\mathrm{fin}}(p)\chi_{D}(p)
  \tau_{\mathbb{Q}}(\sigma)^2 T(\psi)\right)\\
  &\hspace{30mm}\times\int_{-i\infty}^{i\infty}(D\mathbb{N}_{F/\mathbb{Q}}(p\mathfrak{f}))^{-s/2}
  \Lambda(s,(\overline{\sigma}\circ\mathbb{N}_{F/\mathbb{Q}})\overline{\psi})
  (D\mathbb{N}_{F/\mathbb{Q}}(p\mathfrak{f})y)^s ds\\
  =&-\frac{1}{8\pi}\frac{1}{\sqrt{yD\mathbb{N}_{F/\mathbb{Q}}(p\mathfrak{f})}}
  \sigma(-D\mathbb{N}_{F/\mathbb{Q}}(\mathfrak{f}))\psi_{\mathrm{fin}}(p)\tau_{\mathbb{Q}}(\sigma)T(\psi)\\ 
  &\hspace{30mm}\times\int_{-i\infty}^{i\infty}(D\mathbb{N}_{F/\mathbb{Q}}(p\mathfrak{f}))^{-s/2}
  \Lambda(s,(\overline{\sigma}\circ\mathbb{N}_{F/\mathbb{Q}})\overline{\psi})
  (D\mathbb{N}_{F/\mathbb{Q}}(p\mathfrak{f})y)^s ds\\
  =&\sigma(-D\mathbb{N}_{F/\mathbb{Q}}(\mathfrak{f}))\psi_{\mathrm{fin}}(p)T(\psi)f_{\overline{\sigma},\overline{\psi}}
  \left(\frac{i}{D\mathbb{N}_{F/\mathbb{Q}}(p\mathfrak{f})y}\right).
\end{align*}
Thus, by Lemma \ref{bump}, we have Equation \eqref{f;automorphy4}.\\

Case 3. The case where $\sigma$ is trivial.\\
To begin with, we perform the following computation:
\begin{align*}
  \left(f_{\sigma,\psi}+\Theta_{\psi}\right)(z)
  =&\sum_{m\bmod p}\Theta_{\psi}\left|\begin{pmatrix}
    p&m\\
     &p
  \end{pmatrix}\right.(z)\\
  =&\sum_{m\bmod p}\sum_{n\ne 0}a(n)\sqrt{y}K_{\nu}(2\pi |n|y)
  e^{2\pi inx}e^{2\pi in\frac{m}{p}}\\
  =&p\sum_{\substack{n\ne 0\\p\mid n}}a(n)\sqrt{y}K_{\nu}(2\pi |n|y)e^{2\pi inx}\\
  =&p\sum_{\substack{\mathfrak{a}\\p\mid \mathbb{N}_{F/\mathbb{Q}}(\mathfrak{a})}}
  \psi(\mathfrak{a})\sqrt{y}K_{\nu}(2\pi \mathbb{N}_{F/\mathbb{Q}}(\mathfrak{a})y)
  \sin (2\pi \mathbb{N}_{F/\mathbb{Q}}(\mathfrak{a})p^2 x)\\
  =&\psi_{\mathrm{fin}}(p)\Theta_{\psi}(p^2 z)\\
  =&\psi_{\mathrm{fin}}(p)\Theta_{\psi}\left|\begin{pmatrix}
    p&      \\
     &p^{-1}
  \end{pmatrix}\right.(z).
\end{align*}
Thus, we have
\[f_{\sigma,\psi}=-\Theta_{\psi}+\psi_{\mathrm{fin}}(p)\Theta_{\psi}\left|\begin{pmatrix}
  p&     \\
   &p^{-1}
\end{pmatrix}\right. .\]
Using this equation,
\begin{align*}
 f_{\sigma,\psi}\left|\begin{pmatrix}
                            &-1\\
  D\mathbb{N}_{F/\mathbb{Q}}(p\mathfrak{f})&
 \end{pmatrix} \right.
 =&-\Theta_{\psi}\left|\begin{pmatrix}
                              &-1\\
    D\mathbb{N}_{F/\mathbb{Q}}(p\mathfrak{f})&
 \end{pmatrix}\right.
 +\psi_{\mathrm{fin}}(p)\Theta_{\psi}\left|\begin{pmatrix}
  p&      \\
   &p^{-1}
 \end{pmatrix}
 \begin{pmatrix}
                             &-1\\
    D\mathbb{N}_{F/\mathbb{Q}}(p\mathfrak{f})&
 \end{pmatrix}\right.\\
 =&-\Theta_{\psi}\left|\begin{pmatrix}
                               &-1\\
      D\mathbb{N}_{F/\mathbb{Q}}(\mathfrak{f})&
 \end{pmatrix}
 \begin{pmatrix}
  p&     \\
   &p^{-1}
 \end{pmatrix}\right.
 +\psi_{\mathrm{fin}}(p)\Theta_{\psi}\left|\begin{pmatrix}
                             &-1\\
    D\mathbb{N}_{F/\mathbb{Q}}(\mathfrak{f})&
 \end{pmatrix}\right.\\
 =&T(\psi)\Theta_{\overline{\psi}}\left|\begin{pmatrix}
  p&      \\
   &p^{-1}
 \end{pmatrix}\right.
 -\psi_{\mathrm{fin}}(p)T(\psi)\Theta_{\overline{\psi}}\\
 =&T(\psi)\psi_{\mathrm{fin}}(p)\left(-\Theta_{\overline{\psi}}+\overline{\psi}_{\mathrm{fin}}(p)
 \theta_{\overline{\psi}}\left|\begin{pmatrix}
  p&      \\
   &p^{-1}
 \end{pmatrix}\right.\right)\\
 =&T(\psi)\psi_{\mathrm{fin}}(p)f_{\sigma,\overline{\psi}}\\
 =&\sigma(-D\mathbb{N}_{F/\mathbb{Q}}(\mathfrak{f}))
 \psi_{\mathrm{fin}}(p)T(\psi)f_{\sigma,\overline{\psi}}.
\end{align*}
From Cases 1-3, we prove
Equation \eqref{f;automorphy4}, hence Equation \eqref{f;automorphy3} is established.
\end{proof}
Take $m,r\in\mathbb{Z}$ such that $-D\mathbb{N}_{F/\mathbb{Q}}(\mathfrak{f})mr\equiv 1\pmod p$
and $s\in\mathbb{Z}$ such that $sp-D\mathbb{N}_{F/\mathbb{Q}}(\mathfrak{f})mr=1$.
Define
\[\sigma(n)=\begin{cases}
  1\qquad&\text{if }n\equiv m\pmod p\\
  0      &\text{if }n\not\equiv m\pmod p.
\end{cases}\]
Then, 
\[\rho(n)=\begin{cases}
  1\qquad&\text{if }n\equiv m\pmod p\\
  0      &\text{if }n\not\equiv m\pmod p.
\end{cases}\]
Applying Equation (\ref{f;automorphy3}),
\begin{align*}
  \Theta_{\psi}\left|\begin{pmatrix}
    p&m\\
     &p
  \end{pmatrix}
  \begin{pmatrix}
                             &-1\\
   D\mathbb{N}_{F/\mathbb{Q}}(p\mathfrak{f})&
  \end{pmatrix}\right.
  =&T(\psi)\psi_{\mathrm{fin}}(p)\Theta_{\overline{\psi}}\left|\begin{pmatrix}
    p&r\\
     &p
  \end{pmatrix}\right.\\
  =&-\psi_{\mathrm{fin}}(p)\Theta_{\psi}\left|\begin{pmatrix}
                             &-1\\
    D\mathbb{N}_{F/\mathbb{Q}}(\mathfrak{f})&
  \end{pmatrix}\begin{pmatrix}
    p&r\\
     &p
  \end{pmatrix}\right. .
\end{align*}
Thus, it follows that
\begin{align*}
  \Theta_{\psi}\left|\begin{pmatrix}
    p&m\\  &p
  \end{pmatrix}\right.
  =&-\psi_{\mathrm{fin}}(p)\Theta_{\psi}\left|\begin{pmatrix}
                                &-1\\
      D\mathbb{N}_{F/\mathbb{Q}}(\mathfrak{f})&
  \end{pmatrix}
  \begin{pmatrix}
    p&r\\  &p
  \end{pmatrix}
  \begin{pmatrix}
     &-1\\
      D\mathbb{N}_{F/\mathbb{Q}}(p\mathfrak{f})&
  \end{pmatrix}\right.\\
  =&-\psi_{\mathrm{fin}}(p)\Theta_{\psi}\left|\begin{pmatrix}
    p                          &-m\\
    -Dr\mathbb{N}_{F/\mathbb{Q}}(\mathfrak{f})&s
  \end{pmatrix}
  \begin{pmatrix}
    p&m\\  &p
  \end{pmatrix}
  \right. .
\end{align*}
\[\therefore
\Theta_{\psi}=-\psi_{\mathrm{fin}}(p)\Theta_{\psi}\left|\begin{pmatrix}
  p                          &-m\\
  -Dr\mathbb{N}_{F/\mathbb{Q}}(\mathfrak{f})&s
\end{pmatrix}\right.\]

From the same argument in the case $\epsilon=0$, it follows that
\[\Theta_{\psi}\left|\begin{pmatrix}
    a&b\\c&d
  \end{pmatrix}\right.
  =\psi_{\mathrm{fin}}(d)\chi_{D}(d)\Theta_{\psi}\]  
for any 
$\begin{pmatrix}
    a&b\\c&d
\end{pmatrix}\in\Gamma_{0}(D\mathbb{N}_{F/\mathbb{Q}}(\mathfrak{f}))$.\\

Finally we argue on the cuspidality.
Let $l\in\mathbb{Z}$ with $l\geq 2$.
For any Dirichlet character $\sigma :\left(\mathbb{Z}/l\mathbb{Z}\right)^\times \to S^1$,
we define $f_{\sigma,\psi}$ by Equation \eqref{f}.
Assume that $\psi$ is not of the form $\sigma\circ\mathbb{N}_{F/\mathbb{Q}}$.
If $\sigma(-1)=1$, then $f_{\sigma,\psi}=0$.
If $\sigma(-1)=-1$, we denote by $\psi '$ the primitive Hecke character 
induced by $(\sigma\circ\mathbb{N}_{F/\mathbb{Q}})\psi$ 
and by $\mathfrak{f}'$ its conductor.
Since $\psi$ is not equal to $\sigma\circ\mathbb{N}_{F/\mathbb{Q}}$, then $\psi'\ne 1$ hence
$\Lambda(s,\psi')$ is holomorphic on $\mathbb{C}$.
\begin{align*}
f_{\sigma,\psi}(z)
=&\sum_{\substack{m\bmod l\\(m,l)=1}}
\overline{\sigma}(m)\sum_{n\ne 0}a(n)\sqrt{y}K_{\nu}(2\pi |n|y)
e^{2\pi inx}e^{2\pi in\frac{m}{l}}\\
=&\tau_{\mathbb{Q}}(\overline{\sigma})\sum_{n\ne 0}
a(n)\sigma(n)\sqrt{y}K_{\nu}(2\pi |n|y)
e^{2\pi in x}\\
=&-i\tau_{\mathbb{Q}}(\overline{\sigma})\sum_{\mathfrak{a}}
\sigma\circ\mathbb{N}_{F/\mathbb{Q}}(\mathfrak{a})\sqrt{y}K_{\nu}(2\pi\mathbb{N}_{F/\mathbb{Q}}(\mathfrak{a})y)
\cos (2\pi \mathbb{N}_{F/\mathbb{Q}}(\mathfrak{a})x)
\end{align*}
Thus, the Mellin transform is
\begin{align*}
\int_{0}^{\infty}f_{\sigma,\psi}(iy)y^{s-\frac{1}{2}}\frac{dy}{y}
=&-i\tau_{\mathbb{Q}}(\overline{\sigma})\sum_{\mathfrak{a}}
\psi(\mathfrak{a})\circ\mathbb{N}_{F/\mathbb{Q}}(\mathfrak{a})
\int_{0}^{\infty}K_{\nu}(2\pi\mathbb{N}_{F/\mathbb{Q}}(\mathfrak{a})y)y^s\frac{dy}{y}\\
=&-i\tau_{\mathbb{Q}}(\overline{\sigma})\sum_{\mathfrak{a}}
\psi(\mathfrak{a})\sigma\circ\mathbb{N}_{F/\mathbb{Q}}
(\mathfrak{a})(2\pi\mathbb{N}_{F/\mathbb{Q}}(\mathfrak{a}))^{-s}
\Gamma\left(\frac{s-\nu}{2}\right)\Gamma\left(\frac{s+\nu}{2}\right)2^{s-2}\\
=&\frac{1}{4i}\tau_{\mathbb{Q}}(\overline{\sigma})\pi^{-s}
L(s,(\sigma\circ\mathbb{N}_{F/\mathbb{Q}})\psi)
\Gamma\left(\frac{s-\nu}{2}\right)\Gamma\left(\frac{s+\nu}{2}\right)\\
=&\frac{1}{4i}\tau_{\mathbb{Q}}(\overline{\sigma})
(D\mathbb{N}_{F/\mathbb{Q}}(\mathfrak{f}'))^{-s/2}\Lambda(s,\psi ')E(s,(\sigma\circ\mathbb{N}_{F/\mathbb{Q}})\psi)
\end{align*}
where $\displaystyle E(s,(\sigma\circ\mathbb{N}_{F/\mathbb{Q}})\psi)=
\prod_{\substack{\mathfrak{p}\nmid\mathfrak{f}'\\\mathfrak{p}\mid l\mathfrak{f}}}
\left(1-\psi '(\mathfrak{p})\mathbb{N}_{F/\mathbb{Q}}(\mathfrak{p})^{-s}\right)$.\\
By the Mellin inversion formula and Remark \ref{remark;Pharagmen-Lindelöf}, 
we obtain the following asymptotic estimate as $y\rightarrow +0$:
\begin{align*}
  f_{\sigma,\psi}(iy)
  =&-\frac{\sqrt{y}}{8\pi}\tau_{\mathbb{Q}}(\overline{\sigma})
  \int_{R-i\infty}^{R+i\infty}(D\mathbb{N}_{F/\mathbb{Q}}(\mathfrak{f}'))^{-s/2}
  \Lambda(s,\psi ')E(s,(\sigma\circ\mathbb{N}_{F/\mathbb{Q}})\psi)y^{-s}ds\\
  =&O_{t}(y^t)
\end{align*}
for any $t\in \mathbb{R}$ and $R>1$.
Since for $(l,m)=1$, 
$\Theta_{\psi}\left|\begin{pmatrix}
  l&m\\ &l
\end{pmatrix}\right.$
can be expressed as a linear combination of
$\{f_{\sigma,\psi}\}_{\sigma:\left(\mathbb{Z}/l\mathbb{Z}\right)^\times\to S^1}$,
we have
$\displaystyle\Theta_{\psi}\left|\begin{pmatrix}
  l&m\\  &l
\end{pmatrix}\right.(iy)=\Theta_{\psi}\left(iy+\frac{m}{l}\right)=O_{t}(y^t)$
as $y\rightarrow 0$
for any $ t\in \mathbb{R}$. Hence $\Theta_{\psi}$ is a cusp form.
If $\psi$ is $\sigma\circ\mathbb{N}_{F/\mathbb{Q}}$, $\psi '$ is trivial character.
Thus $t\in\mathbb{R}$ has upper bound $-\frac{1}{2}$, 
then it is a Maass wave form but not a cusp form. 
\begin{remark}\label{remark;pf of main result}
  The proof in the case where the narrow class number of $F$ is one is described in \cite{Bump1997}.
  In order to show the cuspidality of $\Theta_{\psi}$, 
  the author claims 
  if $F$ has narrow class number $1$, its discriminant $D$ is a prime number;
  however, it is incorrect.
  If $D$ is not prime, Equation $(9.21)$ of \cite{Bump1997}
  is not sufficient to prove the cuspidality.
\end{remark}
\section{Computation of Petersson inner product} \label{sec;Petesson inner product}
In this section, we give an explicit formula of
Petersson inner product for Maass wave cusp form constructed in Section \ref{sec;main result1}.
\subsection{Proof of Theorem \ref{theorem;explicit computation of Petersson inner product}}
\begin{definition}
  Keep the notation in Section \ref{sec;Introduction}. 
  As same as Section \ref{sec;Introduction},
  for $\Theta_{1},\,\Theta_{2}\in S(\Gamma_{0}(N),\nu,\chi)$, we define 
  Petersson inner product of $\Theta_{1}$ and $\Theta_{2}$ by
  \[\left<\Theta_1, \Theta_2\right>=\int _{\Gamma_{0}(N)\backslash\mathbb{H}}
  \Theta_1(z)\overline{\Theta_{2}(z)}\frac{dxdy}{y^2}.\]
\end{definition}
\vspace{3mm}
Let $\Theta_{1},\Theta_{2}\in M(\Gamma_{0}(N),\nu,\chi)$.
We write the Fourier expansion of $\Theta_{i}$ as: 
\[\Theta_{i}(z)
=\sum_{n\ne 0}a_{i}(n)\sqrt{y}K_{\nu}(2\pi |n|y)e^{2\pi inx}.\]
Moreover, we set
\[I(s)=\int_{0}^{\infty}\int_{0}^{1}
y^{s-2}\Theta_{1}(z)\overline{\Theta}_{2}(z)dxdy.\]
In what follows, we assume that $\mathrm{Re}(s)>1$.
\begin{align*}
  I(s)
  =&\sum_{\substack{n\ne 0\\ m \ne 0}}a_{1}(n)\overline{a_{2}(m)}
  \int_{0}^{\infty}\int_{0}^{1}\int_{0}^{\infty}\int_{0}^{1}
  y^{s-1}\left|K_{\nu}(2\pi |n|y)\right|\left|K_{\nu}(2\pi |m|y)\right|
  e^{2\pi i(n-m)x}dxdy\\
  =&\sum_{n \ne 0} a_{1}(n)\overline{a_{2}(n)} \int_{0}^{\infty}
  \left|K_{\nu}(2\pi |n|y)\right|^{2} y^{s}\frac{dy}{y}\\
  =&(2\pi)^{-s}\sum_{n\ne 0}a_{1}(n)\overline{a_{2}(n)}|n|^{-s}
  \int_{0}^{\infty}\left|K_{\nu}(y)\right|^{2} y^{s}\frac{dy}{y}.
\end{align*}
Here we need the following lemma. 
\begin{lemma} (\cite[p.684]{GradshteynRyzhik2007}) It follows that
  \[\int_{0}^{\infty}\left|K_{\nu}(y)\right|^{2}y^{s}\frac{dy}{y}
  =\frac{2^{-2+s}}{\Gamma(s)}\Gamma\left(\frac{s+2\nu}{2}\right)
  \Gamma\left(\frac{s}{2}\right)^{2}\Gamma\left(\frac{s-2\nu}{2}\right).\]
\end{lemma}
Applying this lemma, we obtain
\begin{equation*}
  I(s)=(2\pi)^{-s}\sum_{n\ne 0}a_{1}(n)\overline{a_{2}(n)}|n|^{-s}
  \frac{2^{-2+s}}{\Gamma(s)}
  \Gamma\left(\frac{s+2\nu}{2}\right)
  \Gamma\left(\frac{s}{2}\right)^{2}\Gamma\left(\frac{s-2\nu}{2}\right).
\end{equation*}
On the other hand, denoting $\Gamma_{\infty}=\left\{ \begin{pmatrix}
1&n\\ &1
\end{pmatrix}
\middle|\, n\in\mathbb{Z}\right\}$,
it follows that
\begin{align*}
  \int_{\Gamma_{0}(N)\backslash\mathbb{H}}
  \frac{1}{2}\sum_{\substack{(c,d)=1\\N\mid c}}\frac{y^s}{|cz+d|^{2s}}
  \Theta_{1}(z)\overline{\Theta_{2}(z)}\frac{dxdy}{y^2}
  =&\int_{\Gamma_{0}(N)\backslash\mathbb{H}}\frac{1}{2}
  \sum_{\gamma\in\Gamma_{\infty}\backslash\Gamma_0(N)}
  (\mathrm{Im} \gamma(z))^{s}\Theta_{1}(z)\overline{\Theta_{2}(z)}\frac{dxdy}{y^2}\\
  =&\int_{\Gamma_{\infty}\backslash\mathbb{H}}(\mathrm{Im}(z))^{s}\Theta_{1}(z)
  \overline{\Theta_{2}(z)}\frac{dxdy}{y^2}\\
  =&I(s).
\end{align*}
Therefore, we have
\begin{align*}
(2\pi)^{-s}\left(\sum_{n\ne 0}a_{1}(n)\overline{a_{2}(n)}
|n|^{-s}\right)\frac{2^{-2+s}}{\Gamma(s)}&\Gamma\left(\frac{s+2\nu}{2}\right)
\Gamma\left(\frac{s}{2}\right)^{2}\Gamma\left(\frac{s-2\nu}{2}\right)\\
=&\int_{\Gamma_{0}(N)\backslash\mathbb{H}}\frac{1}{2}
\sum_{\substack{(c,d)=1\\ N\mid c}}
\frac{y^s}{|cz+d|^{2s}}\Theta_{1}(z)\overline{\Theta_{2}(z)}\frac{dxdy}{y^2}.
\end{align*}
\begin{definition}(Eisenstein series)\\
Let $s\in \mathbb{C}$ and $z\in\mathbb{H}$.
 We define the Eisenstein series $E(s,z)$ by 
 $\displaystyle 
 E(s,z)=\frac{1}{2}\sum_{(c,d)\ne (0,0)}\frac{y^s}{\left|cz+d\right|^{2s}}$.
 The right hand side converges absolutely and uniformly on 
 any compact subset of $\mathrm{Re}(s)>1$. Hence
 it defines a holomorphic function on $\mathrm{Re}(s)>1$. 
\end{definition}
The Eisenstein series has the following properties.
\begin{proposition}(\cite[p66]{Bump1997}) 
It follows that
  \begin{enumerate}
    \item Let $\widehat{E}(s,z)=\pi ^{-s}\Gamma(s)E(s,z)$. Then $\widehat{E}(s,z)$
    is extended to $\mathbb{C}$ meromorphically, has simple poles 
    at $s=0,1$ and is holomorphic at other points.
    Moreover the following properties hold:
    \begin{enumerate}
      \item $\widehat{E}(s,z)=\widehat{E}(1-s,z)$
      \item $\widehat{E}(s,z)=O (y^\sigma)\quad y\rightarrow\infty\, ,\quad
      \sigma\coloneq\mathrm{max}\{\mathrm{Re}(s),1-\mathrm{Re}(s)\}$
      \item $\text{Res}_{s=1}(\widehat{E}(s,z))=\frac{1}{2}$.
    \end{enumerate}
    \item $\widehat{E}(s,z)$ has the following Fourier expansion:
    \[\displaystyle
    \widehat{E}(s,z)=a_{0}(y,s)+\sum_{r\ne 0}a_{r}(y,s)e^{2\pi irx}
    \]
    where
    \begin{equation*}
    \begin{cases}
      a_{0}(y,s)=\pi^{-s}\Gamma(s)\zeta(2s)y^s +\pi^{s-1}\Gamma(1-s)
      \zeta(2-2s)y^{1-s}\qquad &(r=0),\\
      a_{r}(y,s)=2|r|^{s-\frac{1}{2}}\sigma_{1-2s}(|r|)\sqrt{y}
      K_{s-\frac{1}{2}}(2\pi |r|y)&(r\ne 0).
    \end{cases}
  \end{equation*}
  \end{enumerate}
\end{proposition}
\begin{lemma}
  \[\zeta(2s)\prod_{p\mid N}(1-p^{-2s})\cdot\frac{1}{2}
  \sum_{\substack{(c,d)=1\\N\mid c}}\frac{y^s}{|cz+d|^{2s}}
  =N^{-s}\sum_{e\mid N}\frac{\mu(e)}{e}E(s,\frac{N}{e}z)\]
\end{lemma}
\begin{proof}The following direct computation derives the conclusion.
  \begin{align*}
    &N^{-s}\sum_{e\mid N}\frac{\mu(e)}{e}E(s,\frac{N}{e}z)\\
    &=\frac{1}{2}N^{-s}\sum_{e\mid N}\frac{\mu(e)}{e^s}
    \sum_{(c,d)\ne (0,0)}
    \frac{\left(\frac{N}{e}y\right)^{s}}{\left|c\frac{N}{e}z+d\right|^{2s}}
    =\frac{1}{2}N^{-s}\sum_{e\mid N}\mu(e)\sum_{(c,d)\ne (0,0)}
    \frac{(Ny)^s}{|cNz+de|^{2s}}&\\
    &=\frac{1}{2}\sum_{e\mid N}\mu(e)\sum_{(c,d)\ne(0,0)}\frac{y^s}{|cNz+de|^{2s}}
    =\frac{1}{2}\sum_{(c,d')\ne(0,0)}\sum_{\substack{de=d'\\e\mid N}}\mu(e)
    \frac{y^s}{|cNz+d'|^{2s}}&\\
    &=\frac{1}{2}\sum_{(c,d')\ne(0,0)}\sum_{\substack{d\mid d'\\ \frac{d'}{d}\mid N}}
    \mu(\frac{d'}{d})\frac{y^s}{|cNz+d'|^{2s}}
    =\frac{1}{2}\sum_{(c,d')\ne(0,0)}\frac{y^s}{|cdN+d'|^{2s}}
    \sum_{\substack{d\mid d'\\\frac{d'}{d}\mid N}}\mu(\frac{d'}{d})\\
    &=\frac{1}{2}\sum_{(c,d')\ne(0,0)}\frac{y^s}{|cdN+d'|^{2s}}
    \sum_{\substack{d\mid d'\\e\mid N}}\mu(e)
    =\frac{1}{2}\sum_{(c,d')\ne(0,0)}\frac{y^s}{|cNz+d'|^{2s}}
    \sum_{\substack{e\mid (d',N)}}\mu(e)\\
    &=\frac{1}{2}\sum_{\substack{(c,d')\ne(0,0)\\(d',N)=1}}
    \frac{y^s}{|cNz+d'|^{2s}}
    =\frac{1}{2}\sum_{\substack{(c,d\ne (0,0))\\(d,N)=1}}
    \frac{y^s}{|cNz+d|^{2s}}\\
    &=\frac{1}{2}\sum_{n=1}^{\infty}\sum_{\substack{(c,d)=n\\(d,N)=1}}
    \frac{y^s}{|cNz+d|^{2s}}
    =\frac{1}{2}\sum_{n=1}^{\infty}n^{-2s}
    \sum_{\substack{(c,d)\ne(0,0)\\(c,d)=1\\(nd,N)=1}}\frac{y^s}{|cNz+d|^{2s}}\\
    &=\frac{1}{2}\sum_{\substack{n=1\\(n,N)=1}}^{\infty}n^{-2s}
    \sum_{\substack{(c,d)\ne(0,0)\\(c,d)=1\\(d,N)=1}}\frac{y^s}{|cNz+d|^{2s}}
    =\frac{1}{2}\prod_{p\nmid N}\left(1-p^{-2s}\right)\zeta(2s)
    \sum_{\substack{(c,d)\ne(0,0)\\(c,d)=1\\(d,N)=1}}
    \frac{y^s}{|cNz+d|^{2s}}\\
    &=\frac{1}{2}\sum_{n=1}^{\infty}n^{-2s}
    \sum_{\substack{(c,d)\ne(0,0)\\(c,d)=1\\N\mid c}}\frac{y^s}{|cz+d|^{2s}}.
  \end{align*}
\end{proof}

From this Lemma, we have the following. 
\begin{proposition}\label{computation of innerproduct 1}
  Let $\Theta_{1},\Theta_{2}\in M(\Gamma_{0}(N),\nu,\chi)$ and $a_{i}(n)$ be their
  Fourier coefficients respectively. Then for $s\in\mathbb{C}$ with
  $\mathrm{Re}(s)>1$, it follows that
  \begin{align}\frac{1}{4\Gamma(s)}\pi^{-s}\Gamma\left(\frac{s+2\nu}{2}\right)
  &\Gamma\left(\frac{s-2\nu}{2}\right)\Gamma\left(\frac{s}{2}\right)^2
  \left(\sum_{n\ne 0}a_{1}(n)\overline{a_{2}(n)}|n|^{-s}\right)\notag\\
  =&\zeta(2s)^{-1}\prod_{p\mid N}(1-p^{-2s})N^{-s}
  \int_{\Gamma_{0}(N)\backslash\mathbb{H}}\sum_{e\mid N}
  \frac{\mu(e)}{e^s}E(s,\frac{N}{e}z)\Theta_{1}(z)
  \overline{\Theta_{2}(z)}\frac{dxdy}{y^2}\label{eq;computation of innerproduct 1}.
  \end{align}
\end{proposition}
In what follows, we assume that $\psi$ 
is a primitive Hecke character modulo $\mathfrak{f}$ and
that it is not of the form $\chi \circ \mathbb{N}_{F/\mathbb{Q}}$ for any Dirichlet character $\chi$.
Let $\Theta_{\psi}$ be the same as in Section \ref{sec;main result1} and let
$a(n)$ be its  Fourier coefficients:
\begin{equation}
a(n)=
  \begin{cases}
  \frac{1}{2}\sum_{\mathbb{N}_{F/\mathbb{Q}}(\mathfrak{f})=|n|}\psi(\mathfrak{a})
  \qquad&(\text{if }\epsilon=0),\\
  \frac{1}{2i}\text{sgn}(n)\sum_{\mathbb{N}_{F/\mathbb{Q}}(\mathfrak{a})=|n|}\psi(\mathfrak{a})
  &(\text{if }\epsilon=1).
  \end{cases}
\end{equation}
We aim to express $\displaystyle\sum_{n\ne 0}|a(n)|^{2}|n|^{-s}$ in term of Hecke $L$-function.
Putting $a'(n)=\displaystyle\sum_{\mathbb{N}_{F/\mathbb{Q}}(\mathfrak{a})=|n|}\psi(\mathfrak{a})$,
we have
\begin{align*}
  \sum_{n\ne 0}^{\infty}|a(n)|^2|n|^{-s}
  =&2\sum_{n=1}^{\infty}|a(n)|^2 n^{-s}
  =\frac{1}{2}\sum_{n=1}^{\infty}|a'(n)|^2n^{-s}.
\end{align*}
Thus,
\begin{align*}
  \sum_{r=0}^{\infty}a'(p^r)p^{-rs}
  =&\sum_{r=0}^{\infty}\sum_{\mathbb{N}_{F/\mathbb{Q}}(\mathfrak{a})=p^r}
  \psi(\mathfrak{a})\mathbb{N}_{F/\mathbb{Q}}(\mathfrak{a})^{-s}\\
  =&\prod_{\mathfrak{p}\mid p}\sum_{r=0}^{\infty}\psi(\mathfrak{p}^r)
  \mathbb{N}_{F/\mathbb{Q}}(\mathfrak{p})^{-rs}\\
  =&\prod_{\mathfrak{p}\mid p}\left(1-\psi(p)\mathbb{N}_{F/\mathbb{Q}}(\mathfrak{p})^{-s}\right)^{-1}\\
  =&\begin{cases}
    (1-\psi(\mathfrak{p})p^{-2s})^{-1} &\text{if }\left(\frac{D}{p}\right)=-1\\
    (1-\psi(\mathfrak{p})p^{-s})^{-1}(1-\psi(\mathfrak{p}')p^{-s})^{-1}\qquad
    &\text{if }\left(\frac{D}{p}\right)=1\\
    (1-\psi(\mathfrak{p})p^{-s})^{-1}&\text{if }\left(\frac{D}{p}\right)=0
  \end{cases}\\[5mm]
  \eqcolon & (1-\alpha_{p}p^{-s})^{-1}(1-\beta_{p}p^{-s})^{-1}.
\end{align*}
Here $\alpha_{\mathfrak{p}}$, $\beta_{\mathfrak{p}}$
are the Satake parameters defined above.
By a similar computation, we have
$\displaystyle\sum_{r=0}^{\infty}\overline{a'(p^r)}p^{-rs}=(1-\overline{\alpha_{p}}p^{-s})^{-1}
(1-\overline{\beta_{p}}p^{-s})^{-1}$.\\
Considering this computations, we calculate as follows:
\begin{align*}
  &\sum_{n=1}^{\infty}|a'(n)|^{2}n^{-s}
  =\prod_{p}\sum_{r=0}^{\infty}a'(p^r)\overline{a'(p^r)}p^{-rs}\\
  &=\prod_{p}\left(1-|\alpha_{p}|^2|\beta_{p}|^2p^{-2s}\right)
  \left(1-|\alpha_{p}|^2p^{-s}\right)^{-1}
  \left(1-\alpha_{p}\overline{\beta_{p}}p^{-s}\right)^{-1}
  \left(1-\overline{\alpha_{p}}\beta_{p}p^{-s}\right)^{-1}
  \left(1-|\beta_{p}|^2p^{-s}\right)^{-1}\\
  &=\prod_{\left(\frac{D}{p}\right)=0}\begin{cases}
    (1-p^{-s})^{-1}\qquad&\text{if }p\nmid \mathbb{N}_{F/\mathbb{Q}}(\mathfrak{f})\\
    1&\text{if }p\mid\mathbb{N}_{F/\mathbb{Q}}(\mathfrak{f})
  \end{cases}\\
  &\times\prod_{\left(\frac{D}{p}\right)=1}\begin{cases}
    (1-p^{-2s})(1-p^{-s})^{-2}
    (1-\psi(\mathfrak{p})\overline{\psi(\mathfrak{p}')}p^{-s})^{-1}
    (1-\overline{\psi(\mathfrak{p})}\psi(\mathfrak{p}')p^{-s})^{-1}
    \qquad&\text{if }p\nmid \mathbb{N}_{F/\mathbb{Q}}(\mathfrak{f})\\
    (1-|\psi(\mathfrak{p})|^{2}p^{-s})^{-1}
    (1-|\psi(\mathfrak{p}')|^{2}p^{-s})^{-1}
    &\text{if }p\mid \mathbb{N}_{F/\mathbb{Q}}(\mathfrak{f})
  \end{cases}\\
  &\times\prod_{\left(\frac{D}{p}\right)=-1}\begin{cases}
    (1-p^{-2s})(1-p^{-s})^{-2}(1+p^{-s})^{-2}
    \qquad &\text{if }p\nmid \mathbb{N}_{F/\mathbb{Q}}(\mathfrak{f})\\
    1&\text{if }p\mid \mathbb{N}_{F/\mathbb{Q}}(\mathfrak{f})
  \end{cases}\\
  &=\prod_{\substack{p\mid D\\p\nmid \mathbb{N}_{F/\mathbb{Q}}(\mathfrak{f})}}(1-p^{-s})^{-1}
  \times\prod_{\substack{\left(\frac{D}{p}\right)=1\\p\mid \mathbb{N}_{F/\mathbb{Q}}(\mathfrak{f})}}
  (1-|\psi(\mathfrak{p})|^2 p^{-s})^{-1}(1-|\psi(\mathfrak{p}')|^2p^{-s})^{-1}\\
  &\times\!\prod_{\substack{\left(\frac{D}{p}\right)=1\\p\nmid \mathbb{N}_{F/\mathbb{Q}}(\mathfrak{f})}}
  (1-p^{-2s})(1-p^{-s})^{-2}(1-\psi(\mathfrak{p})\overline{\psi(\mathfrak{p}')}p^{-s})^{-1}
  (1-\overline{\psi(\mathfrak{p})}\psi(\mathfrak{p}')p^{-s})^{-1}
  \times\!\prod_{\substack{\left(\frac{D}{p}\right)=-1\\p\nmid \mathbb{N}_{F/\mathbb{Q}} (\mathfrak{f})}}
  (1-p^{-2s})^{-1}\\
  &=\left(\prod_{\substack{\left(\frac{D}{p}\right)=1\\p\mid \mathbb{N}_{F/\mathbb{Q}}(\mathfrak{f})}}
  (1-|\psi(\mathfrak{p})|^{2}p^{-s})^{-1}(1-|\psi(\mathfrak{p}')|^{2}p^{-s})^{-1}\right)
  \left(\prod_{\substack{\left(\frac{D}{p}\right)=1\\p\nmid \mathbb{N}_{F/\mathbb{Q}}(\mathfrak{f})}}
  (1+p^{s})(1-p^{-s})^{-1}\right)\cdot L(s,\psi(\overline{\psi}\circ\sigma)).
  %=&\left(\prod_{\substack{\left(\frac{D}{p}\right)=1\\p\mid\mathbb{N}_{F/\mathbb{Q}}(\mathfrak{f})}}
  %(1-|\psi(\mathfrak{p})|^{2}p^{-s})^{-1}(1-|\psi(\mathfrak{p}')|^{2}p^{-s})
  %(1-p^{-2s})^{-1}(1-p^{-s})^{2}\right)\\
  %&\cdot\left(\prod_{\substack{\left(\frac{D}{p}\right)=1\\p\mid \mathbb{N}_{F/\mathbb{Q}}(\mathfrak{f})}}
  %(1-p^{-2s})(1-p^{-s})^{-2}\right)\cdot L(s,\psi(\overline{\psi}\circ\sigma))\\
  %=&\left(\prod_{\substack{\left(\frac{D}{p}\right)=1\\p\mid \mathbb{N}_{F/\mathbb{Q}}(\mathfrak{f})}}
  %(1-|\psi(\mathfrak{p})|^{2}p^{-s})^{-1}(1-|\psi(\mathfrak{p}')|^{2}p^{-s})^{-1}
  %(1+p^{-s})^{-1}(1-p^{-s})\right)\\
  %&\cdot\left(\prod_{p\mid D}(1+p^{-s})^{-1}(1-\chi_{D}(p)p^{-s})\right)
  %\cdot\frac{\zeta_{F}(s)}{\zeta(2s)}L(s,\psi(\overline{\psi}\circ\sigma))\\
  %=&\left(\prod_{\mathfrak{p}\mid p\mid\mathbb{N}_{F/\mathbb{Q}}(\mathfrak{f})}(1-p^{-s})^{-1}\right)
  %\left(\prod_{p\mid D\mathbb{N}_{F/\mathbb{Q}}(\mathfrak{f})}(1+p^{-s})^{-1}(1-\chi_{D}(p)p^{-s})\right)
  %\frac{\zeta_{F}(s)}{\zeta(2s)}L(s,\psi(\overline{\psi}\circ\sigma))
\end{align*}
Here, we denote the nontrivial element in 
$\mathrm{Gal}(F/\mathbb{Q})$ by $\sigma$.
Moreover,
\begin{align*}
  \prod_{\substack{\left(\frac{D}{p}\right)=1\\p\nmid \mathbb{N}_{F/\mathbb{Q}}(\mathfrak{f})}}
  (1+p^{-s})(1-p^{-s})^{-1}
  =&\left(\prod_{\substack{\left(\frac{D}{p}\right)=1\\p\mid\mathbb{N}_{F/\mathbb{Q}}(\mathfrak{f})}}
  (1+p^{-s})^{-1}(1-p^{-s})\right)
  \left(\prod_{\left(\frac{D}{p}\right)=1}(1+p^{-s})(1-p^{-s})^{-1}\right)\\
  =&\left(\prod_{\substack{\left(\frac{D}{p}\right)=1\\p\mid\mathbb{N}_{F/\mathbb{Q}}(\mathfrak{f})}}
  (1+p^{-s})^{-1}(1-p^{-s})\right)
  \left(\prod_{p\mid D}(1+p^{-s})^{-1}\right)
  \frac{\zeta_{F}(s)}{\zeta(2s)}.
\end{align*}
Therefore, we obtain
\begin{align*}
\sum_{n=1}^{\infty}|a'(n)|^{2}n^{-s}
=&\left(\prod_{\substack{\left(\frac{D}{p}\right)=1\\p\mid\mathbb{N}_{F/\mathbb{Q}}(\mathfrak{f})}}
(1-|\psi(\mathfrak{p})|^{2}p^{-s})^{-1}
(1-|\psi(\mathfrak{p}')|^{2}p^{-s})^{-1}
(1+p^{-s})^{-1}(1-p^{-s})\right)\\
&\hspace{50mm}\times\left(\prod_{p\mid D}(1+p^{-s})^{-1}\right)
\frac{\zeta_{F}(s)}{\zeta(2s)}L(s,\psi(\overline{\psi}\circ\sigma))\\[5pt]
=&\prod_{\substack{\mathfrak{p}\colon\text{split},\,\mathfrak{p}\nmid\mathfrak{f}\\
\mathbb{N}_{F/\mathbb{Q}}(\mathfrak{p})\mid \mathbb{N}_{F/\mathbb{Q}}(\mathfrak{f})}}
(1-\mathbb{N}_{F/\mathbb{Q}}(\mathfrak{p})^{-s})^{-1}
\times\prod_{p\mid D\mathbb{N}_{F/\mathbb{Q}}(\mathfrak{f})}(1+p^{-s})^{-1}(1-\chi_{D}(p)p^{-s})\\
&\hspace{70mm}\times\frac{\zeta_{F}(s)}{\zeta(2s)}L(s,\psi(\overline{\psi}\circ\sigma)).
\end{align*}
Therefore the left hand side of Equation \ref{eq;computation of innerproduct 1}
is given by
\begin{align*}
\frac{1}{8\Gamma(s)}\pi^{-s}\Gamma\left(\frac{s+2\nu}{2}\right)
\Gamma\left(\frac{s-2\nu}{2}\right)&\Gamma\left(\frac{s}{2}\right)^{2}
\left(\prod_{\substack{\mathfrak{p}\colon \text{split},\,\mathfrak{p}\nmid \mathfrak{f}\\
\mathbb{N}_{F/\mathbb{Q}}(\mathfrak{p})\mid \mathbb{N}_{F/\mathbb{Q}}(\mathfrak{f})}}
(1-\mathbb{N}_{F/\mathbb{Q}}(\mathfrak{p})^{-s})^{-1}\right)\\
&\times\left(\prod_{p\mid D\mathbb{N}_{F/\mathbb{Q}}(\mathfrak{f})}
(1+p^{-s})^{-1}(1-\chi_{D}(p)p^{-s})\right)
\frac{\zeta_{F}(s)}{\zeta(2s)}L(s,\psi(\overline{\psi}\circ\sigma)).
\end{align*}
The both side of the equation in
Equation \ref{eq;computation of innerproduct 1}
are meromorphic on $\mathbb{C}$, hence the equation is valid for any $s\in \mathbb{C}$.
Taking residue at $s=1$ in Equation \ref{eq;computation of innerproduct 1}
and applying the Möbius inversion formula, 
we obtain Theorem \ref{theorem;explicit computation of Petersson inner product}:
\small
\begin{align*}
  \left<\Theta_{\psi},\Theta_{\psi}\right>
  =\frac{1}{4\pi}\phi&(D\mathbb{N}_{F/\mathbb{Q}}(\mathfrak{f}))^{-1}(D\mathbb{N}_{F/\mathbb{Q}}(\mathfrak{f}))^{2}
  \Gamma\left(\frac{1+2\nu}{2}\right)
  \Gamma\left(\frac{1-2\nu}{2}\right)\\
  &\times\left(
    \prod_{\substack{\mathfrak{p}\colon\text{split},\mathfrak{p}\nmid\mathfrak{f}\\
    \mathbb{N}_{F/\mathbb{Q}}(\mathfrak{p})\mid\mathbb{N}_{F/\mathbb{Q}}(\mathfrak{f})}}
    (1-\mathbb{N}_{F/\mathbb{Q}}(\mathfrak{p})^{-1})^{-1}\right)
    \times\left(\prod_{p\mid D\mathbb{N}_{F/\mathbb{Q}}(\mathfrak{f})}(1-p^{-1})(1-\chi_{D}(p)p^{-1})\right)\\
  &\hspace{65mm}\times\mathrm{Res}_{s=1}(\zeta_{F}(s))L(1,\psi(\overline{\psi}\circ\sigma))
\end{align*}
\normalsize
where $\sigma$ is the nontrivial element
of $\text{Gal}(F/\mathbb{Q})$ and $\phi$ is the Euler function.\\
\begin{remark}
  It is known that a Hecke character $\psi$ on $F$ can be written as 
  $\chi\circ\mathbb{N}_{F/\mathbb{Q}}$ for some Dirichlet character $\chi$
  if and only if $\psi=\psi\circ\sigma$ for all 
  $\sigma\in\mathrm{Gal}(F/\mathbb{Q})$.
  Since we assume that $\psi$ is not of the form $\chi\circ\mathbb{N}_{F/\mathbb{Q}}$,
  then $\psi(\overline{\psi}\circ\sigma)\ne 1$ and 
  $L(s,\psi(\overline{\psi}\circ\sigma))$ is holomorphic at $s=1$.
\end{remark}

\subsection{Examples}\label{SomeExamples}
We already derived a general formula for
$\langle \Theta_\psi, \Theta_\psi \rangle$,
and now we compute it explicitly in concrete examples.
\subsubsection{The dihedral case in general}
If $L/F$ is a cyclic extension of degree $n\ge 2$, by class field theory,
we have
\[J_{F}^{\mathfrak{f}}/P_{F,+} \mathbb{N}_{L/F}(J_{L}^{\mathfrak{f}})
\xrightarrow[\cong]{\text{rec}_{F}} \mathrm{Gal}(L/F)\]
where $P^{\mathfrak{f}}_{F,+}=
\{(a)\mid a\in F^\times,\,a\equiv 1 \mod \mathfrak{f},\,\phi (a)>0 \text{ for any real embedding }\phi \}$
and $\mathfrak{f}$ is the conductor of $L/F$.
Taking a primitive character $\psi '$ of $\text{Gal}(L/F)$, we define
a primitive Hecke character $\psi$ modulo $\mathfrak{f}$ by the following composition:
\[\psi\colon J^{\mathfrak{f}}_{F}\xrightarrow{\text{canonical map}}
J_{F}^{\mathfrak{f}}/P^{\mathfrak{f}}_{F,+} \mathbb{N}_{L/F}(J_{L}^{\mathfrak{f}})
\xrightarrow[\cong]{\text{rec}_{F}} \mathrm{Gal}(L/F)
\xrightarrow{\psi'}S^1.\]
Suppose that $\psi$ is not of the form $\chi\circ\mathbb{N}_{F/\mathbb{Q}}$ for
any Dirichlet character $\chi$.
Also, by the uniqueness of class fields, $L/\mathbb{Q}$ is a Galois extension.
We assume that its Galois group $\mathrm{Gal}(L/\mathbb{Q})$ 
is isomorphic to the dihedral group $D_n$ of order $2n$.
If $n$ is prime and $L/\mathbb{Q}$ is not abelian, then this is the case. 
We remark that the type of $\psi$ is $(0,0,0,0)$ if $L$ is totally real and $(1,1,0,0)$ if $L$ is 
totally imaginary.
%Thus, applying Theorem \ref{theorem;explicit computation of Petersson inner product},
%we obtain
%\begin{align*}
%  \left<\Theta_{\psi},\Theta_{\psi}\right>
%  =&\phi(D\mathbb{N}_{F/\mathbb{Q}}(\mathfrak{f}))^{-1}(D\mathbb{N}_{F/\mathbb{Q}}(\mathfrak{f}))^{2}
%    \times\left(
%    \prod_{\substack{\mathfrak{p}\colon\text{split},\mathfrak{p}\nmid\mathfrak{f}\\
%    \mathbb{N}_{F/\mathbb{Q}}(\mathfrak{p})\mid\mathbb{N}_{F/\mathbb{Q}}(\mathfrak{f})}}
%    (1-\mathbb{N}_{F/\mathbb{Q}}(\mathfrak{p})^{-1})^{-1}\right)\\
%    &\times\left(\prod_{p\mid D\mathbb{N}_{F/\mathbb{Q}}(\mathfrak{f})}(1-p^{-1})(1-\chi_{D}(p)p^{-1})\right)
%   \times\frac{h_{F}R_{F}}{w_{F}\sqrt{D}}L(1,\psi(\overline{\psi}\circ\sigma))
%  \end{align*}
% where $h_F$, $R_F$, $w_F$, and $D$ denote the class number, the regulator, 
% the number of roots of unity, and
% the discriminant of $F$ respectively. 
\subsubsection{The case when $F=\mathbb{Q}(\sqrt{229})$}
For example, we consider $F=\mathbb{Q}(\sqrt{229})$.
Let $L$ be the Hilbert class field of $F$.
Since the class number of $F$ is $h_{F}=3$, we have $[L:F]=3$.
Let $\alpha$ be a root of polynomial $X^{3}-4X-1$ (cf.\cite[p.1242]{CohenRoblot1999}). 
Note that all roots of
$X^{3}-4X-1$ are real.
Putting $K=\mathbb{Q}(\alpha)$,
we have $L=KF$.
Note that $\mathrm{Gal}(L/\mathbb{Q})\cong D_3(\cong S_3)$ 
and $L$ is a totally real field.\\
Denote by $\mathcal{C}_{F}$ ideal class group of $F$.
We take a primitive character $\psi'$ of $\mathrm{Gal}(L/F)$ and define
the Hecke character $\psi$ by the following composition:
\[\psi\colon J^{(1)}_{F}\xrightarrow{\text{canonical}}\mathcal{C}_{F}\xrightarrow[\cong]{\text{rec}_{F}}\text{Gal}(L/F)\xrightarrow{\psi'}S^1\]
where the first map stands for the canonical surjection.
Under this setting,
we compute as follows:
\begin{align*}
  L(s,\psi(\overline{\psi}\circ\sigma))
  =&\prod_{\mathfrak{p}}(1-\psi(\mathfrak{p})^{2}\mathbb{N}_{F/\mathbb{Q}}(\mathfrak{p})^{-s})^{-1}
  =L(s,L/F,\psi'^{2})\\
  =&L(s,L/\mathbb{Q},\text{Ind}_{\text{Gal}(L/F)}^{\text{Gal}(L/\mathbb{Q})}\psi '^{2})
  =L(s,L/\mathbb{Q},\rho)\\
  =&\frac{L(s,L/\mathbb{Q},\text{Ind}_{\text{Gal}(L/K)}^{\text{Gal}(L/\mathbb{Q})}\mathbf{1})}
  {L(s,L/\mathbb{Q},\mathbf{1})}
  =\frac{\zeta_{K}(s)}{\zeta(s)}.
\end{align*}
Here, $\mathbf{1}$ stands for the trivial character and $\rho$ is the two-dimensional irreducible representation of 
$\text{Gal}(L/\mathbb{Q})\cong S_{3}$. 
Thus we have $L(1,\psi(\overline{\psi\circ\sigma}))=\text{Res}_{s=1}(\zeta_{K}(s))$.
From the above, the value of Petersson inner product of $\Theta_{\psi}$ is
given by:
\[\left<\Theta_{\psi},\Theta_{\psi}\right>
=6R_{F} R_{K}\approx 38.3345331336184\cdots\]
where $R_{F}$, $R_{K}$ are the regulators of $F$, $K$.
\subsubsection{The case when $F=\mathbb{Q}(\sqrt{445})$}
For the second example, we consider $F=\mathbb{Q}(\sqrt{445})$.
Then the ideal class group $\mathcal{C}_F$ is isomorphic to $\mathbb{Z}/4\mathbb{Z}$.
Let $L$ be the Hilbert class field of $F$.
Set fields $K_1=\mathbb{Q}(\alpha)$ and $K_2=\mathbb{Q}(\sqrt{5},\sqrt{89})$ where
$\alpha$ is a root of polynomial $X^4-X^3-5X^2+2X+4$.
Then $L$ is the composition of the fields $F$ and $K_1$ (cf. \cite[p.1242]{CohenRoblot1999}), 
and $\mathrm{Gal}(L/\mathbb{Q})$ is isomorphic to $D_4$.
Note that $L$ is a totally real field. Taking a primitive character $\psi'$ 
of $\mathrm{Gal}(L/\mathbb{Q})$, we define a Hecke character given by the 
following composition:
\[J^{(1)}_{F}\xrightarrow{\text{canonical}}\mathcal{C}_{F}
\xrightarrow[\cong]{\text{rec}_{F}}\mathrm{Gal}(L/F)\xrightarrow{\psi'}S^1.\]
In this case, $L(s,\psi(\overline{\psi}\circ\sigma))=L(s,\psi^2)$ 
is written in terms of Dedekind zeta functions:
\begin{align*}
  L(s,\psi^2)=&L(s,\psi'^2,L/F)
  =L(s,\mathrm{Ind}_{\mathrm{Gal}(L/F)}^{\mathrm{Gal}(L/\mathbb{Q})}\psi'^2,L/\mathbb{Q})\\
  =&\frac{L(s,\mathrm{Ind}_{\mathrm{Gal}(L/K_2)}^{\mathrm{Gal}(L/\mathbb{Q})}\mathbf{1},L/\mathbb{Q})}
  {L(s,\mathrm{Ind}_{\mathrm{Gal}(L/F)}^{\mathrm{Gal}(L/\mathbb{Q})}\mathbf{1},L/\mathbb{Q})}
  =\frac{L(s,\mathbf{1},L/K_2)}{L(s,\mathbf{1},L/F)}
  =\frac{\zeta_{K_2}(s)}{\zeta_{F}(s)}.
\end{align*}
Here, we use the fact that 
$\mathrm{Ind}_{\mathrm{Gal}(L/K_2)}^{\mathrm{Gal}(L/\mathbb{Q})}\mathbf{1}
\cong
\mathrm{Ind}_{\mathrm{Gal}(L/F)}^{\mathrm{Gal}(L/\mathbb{Q})}\mathbf{1}\bigoplus 
\mathrm{Ind}_{\mathrm{Gal}(L/F)}^{\mathrm{Gal}(L/\mathbb{Q})}\psi'^2$
are isomorphic as representations. This isomorphism can be confirmed
by computing of associated characters. 
Thus we have
\[\left<\Theta_{\psi},\Theta_{\psi}\right>
=4R_{K_2}\approx81.0223272397348\cdots\]
where $R_{K_2}$ is the regulator of $K_2$.
\subsubsection{The case when $F=\mathbb{Q}(\sqrt{401})$}
For the third example, we consider $F=\mathbb{Q}(\sqrt{401})$.
Let $L$ be the Hilbert class field of $F$.
Since the class number of $F$ is $h_F=5$, we have $[L:F]=5$.
By taking a root $\alpha$ of $X^5-X^4-5X^3+4X^2+3X-1$
(cf.\cite[p.1243]{CohenRoblot1999}),
the field $L$ can be written as $L=FK$ where $K=\mathbb{Q}(\alpha)$.
Note that $L$ is a totally real field.
Now we take a non-trivial character 
$\psi'\colon \mathrm{Gal}(L/F)\cong \mathbb{Z}/5\mathbb{Z}\to S^1$ 
of $\mathrm{Gal}(L/F)$
and define a Hecke character $\psi$ by the following composition:
\[J^{(1)}_{F}\xrightarrow{\text{canonical}}\mathcal{C}_{F}
\xrightarrow[\cong]{\text{rec}_{F}}\mathrm{Gal}(L/F)\xrightarrow{\psi'}S^1.\]
Then, we have $L(s,\psi(\overline{\psi}\circ\sigma))=L(s,\psi'^2,L/F)$.
Applying Theorem \ref{theorem;explicit computation of Petersson inner product},
we obtain
\[\left<\Theta_{\psi},\Theta_{\psi}\right>
=\frac{5}{2}\sqrt{401}R_{F}\cdot L(1,\psi'^2,L/F)\]
where $R_{F}$ is the regulator of $F$.
It is a difficult task to give the special value of 
$L(s,\psi(\overline{\psi}\circ\sigma))$ at $s=1$ in terms of 
number fields invariants
since there are non-rational representations of $D_5$.
However there is a following isomorphism as representations of $\mathrm{Gal}(L/\mathbb{Q})$:
\[\mathrm{Ind}_{\mathrm{Gal(L/K)}}^{\mathrm{Gal}(L/\mathbb{Q})}\mathbf{1}\cong
\mathrm{Ind}_{\mathrm{Gal}(L/F)}^{\mathrm{Gal}(L/\mathbb{Q})}\psi'^2\bigoplus 
\mathrm{Ind}_{\mathrm{Gal}(L/F)}^{\mathrm{Gal}(L/\mathbb{Q})}\psi'^4\bigoplus
\mathbf{1}.\]
Therefore we have $\displaystyle L(1,\psi'^2,L/F)\cdot L(1,\psi'^4,L/F)=\frac{\zeta_{K}(s)}{\zeta(s)}$.
Then the following relation holds:
\[\left<\Theta_\psi,\Theta_\psi\right>\left<\Theta_{\psi^2},\Theta_{\psi^2}\right>
=100R_F^2\cdot R_K\approx12489.3392834563\cdots\]
where $R_K$ is the regulator of $K$.
Similar relations can be obtained if the class number $h_F$ is prime and
$K$ is not Galois over $\mathbb{Q}$.

\begin{remark}
  If $F$ has class number $h_F=2$, we cannot give an example
  similar to the above cases by taking the Hilbert class field. To explain this, we consider the case $F=\mathbb{Q}(\sqrt{10})$.
  Its class number is $h_F=2$ and 
  its Hilbert class field is $L=\mathbb{Q}(\sqrt{5},\sqrt{2})$.
  Define a Hecke character $\psi$ as follows:
  \[\psi:J^{(1)}_{F}\xrightarrow{\text{canonical}}\mathcal{C}_F\xrightarrow[\cong]{\mathrm{rec}_F}\mathrm{Gal}(L/F)
  \xrightarrow{\text{the non-trivial character}} \{\pm 1\}.\]
  %where $\mathcal{C}_F$ is the ideal class group of $F$.
  By elementary arguments, the Hecke character 
  $\displaystyle\psi_1=\left(\frac{\cdot}{5}\right)\circ\mathbb{N}_{F/\mathbb{Q}}$
  has the conductor $1$ and the primitive Hecke character induced by $\psi_1$ 
  is equal to $\psi$.
  Thus, $\Theta_\psi$ is not cuspidal. 
\end{remark}

\end{document}